\documentclass[a4paper,12pt]{article}
\usepackage[hidelinks]{hyperref}
\usepackage{pifont}
\usepackage{amsmath,amsthm,amssymb,amsfonts}
\usepackage{longtable}
\usepackage{multirow}
\usepackage{leftidx}
\usepackage[noend]{algpseudocode}
\usepackage{algorithmicx,algorithm}
\usepackage{bbm,bm}
\usepackage{latexsym}
\usepackage{mathrsfs}
\usepackage{threeparttable}
\usepackage{tabularx}
\usepackage{booktabs}
\usepackage{multicol,graphics}
\usepackage[dvips]{graphicx}
\usepackage{rotating}
\usepackage{epstopdf}
\usepackage{epsfig}
\usepackage{subfigure}
\usepackage{indentfirst}
\usepackage{geometry}
\usepackage{caption}
\usepackage{lineno}
\usepackage{enumerate,mdwlist}
\usepackage[numbers,sort&compress]{natbib}  
\usepackage[utf8]{inputenc}
\usepackage{float}
\theoremstyle{definition}
\newtheorem{The}{Theorem}[section]    
\newtheorem{Lem}[The]{Lemma}
\newtheorem{Cor}[The]{Corollary}
\newtheorem{Pro}[The]{Proposition}

\newtheorem{Cla}{Claim}
\newtheorem{Subcla}{Claim}[subsection]

\newtheorem{Con}[The]{Conjecture}
\newtheorem{Cas}{Case}

\numberwithin{equation}{section}

\begin{document}
\title{A chain theorem for planar PM-compact bricks}
\author{Yaxian Zhang$^1$, Qiuli Li$^2$,  Fuliang Lu$^3$\thanks{Corresponding author.\newline
   {\small E-mails:\ yaxianzhang@sicnu.edu.cn(Y. Zhang), qlli@lzu.edu.cn(Q. Li), flianglu@163.com(F. Lu)}}}
\date{
   {\small 1. School of Mathematical Sciences, Sichuan Normal University, Chengdu, 610068,  China}\\
   {\small 2. School of Mathematics and Statistics, Lanzhou University, Lanzhou, 730000, China}\\
  {\small 3. School of Mathematics and Statistics $\&$ Fujian Key Laboratory of Granular Computing and Applications, Minnan Normal University, Zhangzhou, 363000, China}
}
\maketitle
\begin{abstract}
   A connected graph $G$ is matching covered if it contains at least one edge and every edge belongs to some perfect matching of $G$.
   A matching covered graph is a brick if it is 3-connected and after removing any two vertices, the resulting graph has a perfect matching.   Bricks are the building blocks of matching covered graphs.

A brick $G$ is extremal if the incidence vectors of all perfect matchings are linearly independent over $\mathbb{R}$.
 Zhou et al. [On planar extremal nonsolid bricks, J. Graph Theory, 2026] proved that for every  planar extremal  brick $G$, there exists a sequence of  planar extremal  bricks $G_0,G_1,\dots,G_n$ such that the underlying graph of $G_n$ is an odd wheel, $G_0=G$, and  each $G_i$ $(0\leq i < n)$ has a triangle $T_i$ such that $G_i/T_i = G_{i+1}$, where $G_i/T_i$ means contracting the triangle $T_i$ into a single vertex.
 
 In this paper, the result of Zhou et al. is generalized to planar PM-compact bricks, a larger class of bricks including extremal bricks. We present a chain theorem for planar PM-compact bricks: for every planar PM-compact brick $G$, there exists a sequence of planar PM-compact bricks $G_0,G_1,\dots,G_n$ such that the underlying graph of $G_n$ is an odd wheel, $G_0=G$, and each $G_i$ $(0\leq i < n)$ contains a triangle $T_i$ such that $G_i/T_i = G_{i+1}$.
 \vskip 0.1 in
 {
 \noindent {\bf Keywords:} \ Matching covered graph; PM-compact brick; Extremal brick; Chain theorem
 }\medskip
\end{abstract}
\section{Introduction}
All graphs in this paper are finite and contain no loops (multiple edges are allowed).
Given a graph, its \emph{underlying graph} is the simple graph obtained by replacing each set of multiple edges with a single edge.
Let $G$ be a connected graph with vertex set $V(G)$ and edge set $E(G)$.
An edge subset of $E(G)$ is called a \emph{matching} of a graph $G$ if no two edges in the subset share a common vertex.
A matching $M$ of a graph $G$ is called \emph{perfect} if $|M|=\frac{|V(G)|}{2}$. A graph is \emph{matchable} if it has a perfect matching. In particular, the graph with no vertices is regarded as matchable.
A connected graph is \emph{matching covered} if it contains at least one edge,
and each edge belongs to some perfect matching.
For the terminology that is specific to matching covered graphs,
we follow Lov\'asz and Plummer \cite{Lovasz1986}.

For a nonempty subset $X\subset V(G)$,
we call $\partial_G(X)$ (simply $\partial(X)$ if $G$ is understood) an \emph{edge cut} of $G$, which is the set of edges with one end vertex
in $X$, the other in $\overline{X}=V(G)\setminus X$.
In particular, if $|X|=1$ or $|\overline{X}|=1$, then $\partial(X)$ is called \emph{trivial}; otherwise, it is \emph{nontrivial}.
An edge cut of $G$ is called a \emph{tight cut},
if it intersects each perfect matching of $G$ at exactly one edge.
A matching covered graph without nontrivial tight cuts is a \emph{brick},
if it is nonbipartite, and a \emph{brace} otherwise.

A graph $G$ with at least four vertices is called \emph{bicritical}, if for any two distinct vertices $x$ and $y$ in $G$, $G-x-y$ is matchable.
Edmonds et al. \cite{Edmonds1982} showed that
a graph $G$ is a brick if and only if it is 3-connected and bicritical (see also Szigeti \cite{Szigeti2002} and de Carvalho et al. \cite{Carvalho2018}).
Lov\'asz \cite{Lovasz1987} proved that any matching
covered graph can be decomposed into a unique list of bricks and braces
up to multiple edges by a procedure called the tight cut decomposition.
Precisely, two different tight cut decompositions of a matching covered graph $G$ yield the same number of bricks, which is
called the brick number of $G$ and denoted by $b(G)$.

The \emph{incidence vector} of a perfect matching $M$ in a graph $G$ is a vector in $\{0,1\}^{E(G)}$ with entry $1$ for $e\in M$ and $0$ otherwise.
A matching covered graph $G$ is called \emph{extremal} if the incidence vectors of all perfect matchings are linearly independent over $\mathbb{R}$.
For an odd integer $k\geq 3$, an \emph{odd wheel} $W_k$ is the graph obtained from a cycle $R$ of length $k$ by adding a new vertex $h$ and joining it to all the vertices of $R$. The cycle $R$ is the \emph{rim} of $W_k$,
the vertex $h$ is its \emph{hub}, and the edges incident with $h$ are its \emph{spokes}.
Let $K_n$ be a complete simple graph with $n$ vertices. Then $W_3\cong K_4$.
Note that each vertex of $W_3$ can be regarded as its hub, and the odd wheel $W_k$ with $k\geq 5$ has a unique hub.
Odd wheels are extremal.

\begin{figure}[h]
   \centering
   \includegraphics[scale=0.35]{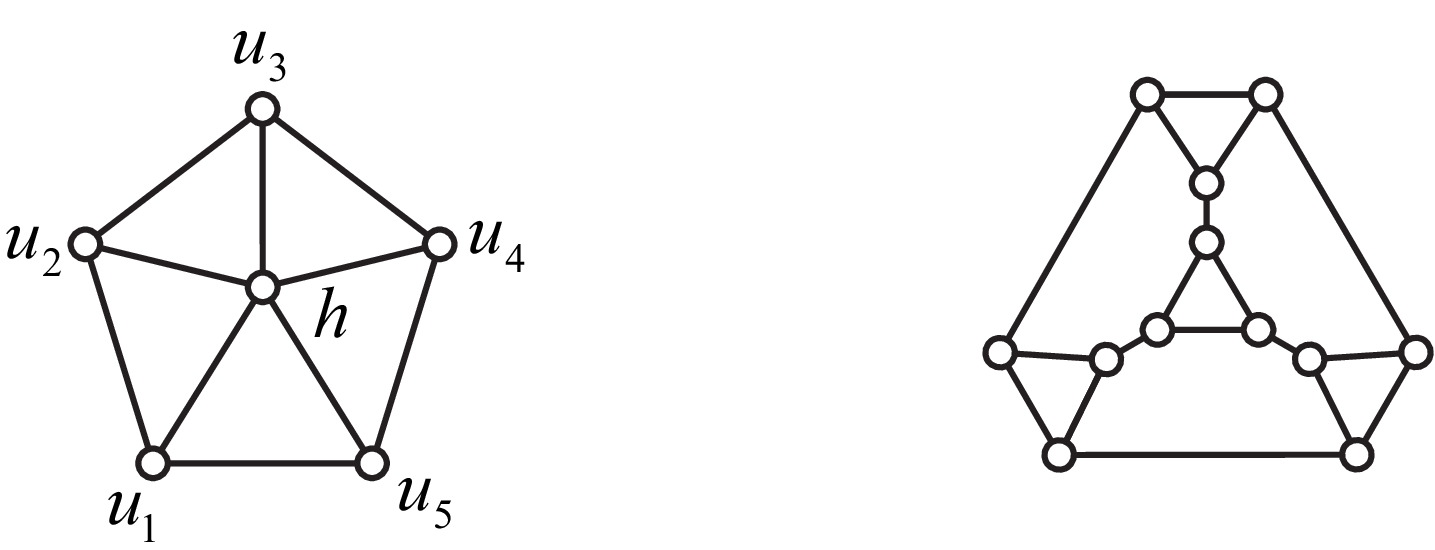}
   \caption{\label{f2}Left: an odd wheel $W_5$; Right: the graph $K_4^t$.}
\end{figure}

When $H$ is a subgraph of a graph $G$, we use $H$ to denote both the subgraph and its vertex set; thus, both $G/V(H)$ and $G/H$ denote the graph obtained by contracting $H$ to a single vertex, deleting any resulting loops while retaining all multiple edges.
De Carvalho et al.  conjectured a chain theorem for  extremal bricks, and showed the conjecture holds for cubic bricks \cite{Carvalho2004}.

\begin{Con}
   \label{Con1}
   (\cite{Carvalho2004})
   For every extremal brick $G$ other than the Petersen graph, there exists a sequence of extremal  bricks $G_0,G_1,\dots,G_n$ such that the underlying graph of $G_n$ is an odd wheel, $G_0=G$, and  each $G_i$ $(0\leq i < n)$ has a triangle $T_i$ such that  $G_i/T_i = G_{i+1}$.
\end{Con}

Feng and Lu \cite{FengLu2021} constructed an infinite family of nonplanar counterexamples to Conjecture \ref{Con1}. Zhou et al. \cite{Zhou2026} showed that Conjecture \ref{Con1} holds for planar extremal bricks (note that the Petersen graph is not planar).
\begin{The}
   \label{Th-extremal}
   (\cite{Zhou2026})
   For every planar extremal brick $G$, there exists a sequence of planar extremal  bricks $G_0,G_1,\dots,G_n$ such that the underlying graph of $G_n$ is an odd wheel, $G_0=G$, and  each $G_i$ $(0\leq i < n)$ has a triangle $T_i$ such that  $G_i/T_i = G_{i+1}$.
\end{The}

A matchable graph $G$ is \emph{PM-compact} if for each pair of vertex-disjoint even cycles $C_1$ and $C_2$, $G-V(C_1)-V(C_2)$ is not matchable. For background on PM-compact graphs, see Balinski and Russakoff \cite{Balinksi1974} and Chv\'atal \cite{Chvatal1975}. Several characterizations of special PM-compact graphs are known.
 Wang et al. \cite{WXM2013} characterized PM-compact graphs that are bipartite (see the following Theorem \ref{3.1}) or near-bipartite, whereas Wang et al. \cite{WXM2014} presented all the PM-compact graphs that are cubic and claw-free.
 Carvalho et al. \cite{Carvalho2020} got the PM-compact Birkhoff-von Neumann graphs (including  solid bricks).
 From the definition of PM-compact graphs, we can directly obtain the following two properties.
\begin{Pro}
   \label{PM-P1}
   Let $H$ be a matchable subgraph of a PM-compact graph $G$. If $G-V(H)$ is matchable, then $H$ is PM-compact.
\end{Pro}
Let $M$ be a perfect matching of a graph $G$. An \emph{$M$-alternating path} (resp., \emph{$M$-alternating cycle}) is a path (resp.,  a cycle) whose edges alternate between $M$ and $E(G)\setminus M$.

\begin{Pro}
   \label{PM-P2}
   Let $M$ be a perfect matching of a PM-compact graph $G$. Then $G$ does not contain two vertex-disjoint $M$-alternating cycles.
\end{Pro}

The inclusion relationship between extremal graphs and PM-compact graphs was established in \cite{Deng2021}. We provide a shorter proof here.
\begin{Pro}
   (\cite{Deng2021})
   \label{EAP}
Every extremal graph is PM-compact.
\end{Pro}
\begin{proof}

    Suppose, to the contrary,  that there exists an extremal graph $G$  that is not PM-compact.
    Then $G$ contains two vertex-disjoint even cycles $C_1$ and $C_2$ such that $G-V(C_1)-V(C_2)$ has a perfect matching, say $M$.
   Assume that $M_1$ and $M_2$ are the two perfect matchings of $C_1$, and $M_3$ and $M_4$ are the two perfect matchings of $C_2$.
    Let
    \[
    N_1 = M \cup M_1 \cup M_3,\quad
    N_2 = M \cup M_2 \cup M_4,\quad
    N_3 = M \cup M_1 \cup M_4,\quad
    N_4 = M \cup M_2 \cup M_3.
    \]
    For each \(i = 1,2,3,4\), let \(\chi_i\in\mathbb{R}^{E(G)}\) be the characteristic vector of \(N_i\). Then
  $\chi_1 + \chi_2 = \chi_3 + \chi_4,$
    which implies that \(\chi_1, \chi_2, \chi_3, \chi_4\) are linearly dependent.
    However, the characteristic vectors of all perfect matchings of an extremal graph are linearly independent.
    This contradiction shows that our assumption is false; hence the result follows.
    \end{proof}

The converse of Proposition \ref{EAP} does not necessarily hold.
Let $K_4^t$  be the cubic graph obtained from $K_4$ by replacing each vertex with a triangle (see  Fig. \ref{f2}). The graph $K_4^t$ is a brick and has 8 perfect matchings.
Also, $K_4^t$ is PM-compact:  it does not contain two vertex-disjoint even cycles.
But $K_4^t$ is not extremal: by Edmonds et al. \cite{Edmonds1982}, the dimension of its perfect matching polytope is $|E(K_4^t)|-|V(K_4^t)|+1-b(K_4^t)$, i.e. six. Since this polytope is the convex hull of the incidence vectors of the eight perfect matchings of $K_4^t$, these eight vectors are affinely, and hence linearly, dependent.

A sequence of bricks linked by successive triangle contractions is called a \emph{chain}. Our main result is the following chain theorem.

\begin{The}
   \label{MaT}
   Let $G$ be a planar PM-compact brick. Then there exists a sequence of planar PM-compact bricks $G_0,G_1,\dots,G_n$ such that $G_0=G$, the underlying graph of $G_n$ is an odd wheel, and each $G_i$ ($0\leq i < n$) has a triangle $T_i$ such that $G_i/T_i = G_{i+1}$.
\end{The}
Theorem \ref{MaT} generalizes the chain theorem for planar extremal bricks (Theorem \ref{Th-extremal}) to the broader class of planar PM-compact bricks.

Section 2 introduces the definitions and preliminary results used throughout the paper. Section 3 develops auxiliary results on graphs with few perfect matchings and on bipartite PM-compact graphs, which are used in the subsequent structural analysis. Based on these results, Section 4 establishes a structure theorem for planar PM-compact bricks. Section 5 then analyzes splicings involving odd wheels and identifies the special configurations needed for the inductive argument. Finally, Section 6 combines the structure theorem with the splicing results to prove Theorem \ref{MaT}.
\section{Preliminaries}

For an edge set $X$ and a vertex set $S$ in a graph $G$, we denote by $G[X]$ and $G[S]$ the subgraphs induced by $X$ and $S$, respectively.
We denote by $E_G[S]$ (or simply $E[S]$) the set of edges with both end vertices in $S$.
We denote by $N_G(S)$ (or simply $N(S)$) the set of vertices in $V(G)\setminus S$ that are adjacent to at least one vertex in $S$. When $S = \{v\}$, we write $N_G(v)$ for $N_G(\{v\})$.  Let $u$ be a vertex of a graph $G$. The \emph{degree} of $u$, denoted by $d_G(u)$, is the number of edges incident with $u$, counted with multiplicity.
Let $\delta(G)$ and $\Delta(G)$ denote the minimum and maximum degrees of $G$, respectively.

For two disjoint nonempty vertex subsets $S_1$ and $S_2$ of $G$, let $E_G[S_1,S_2]$ (or simply $E[S_1,S_2]$) denote the set of edges between $S_1$ and $S_2$.
In particular, when $S_1\cup S_2=V(G)$, the set $E[S_1,S_2]$ is an edge cut of $G$.
For two disjoint subgraphs $H_1$ and $H_2$, we write $E[H_1,H_2]$ for $E[V(H_1),V(H_2)]$.

For a path $P$, the two vertices of degree one in $P$ are its {\em end vertices}, and the other vertices of $P$ are its {\em internal vertices}.
For two different vertices $u $ and $v$ in $V(P)$, $uPv$ denotes the subpath of $P$ with end vertices $u$ and $v$.
Let $C$ be a cycle in a plane graph, $u_i $ and $ u_{i+1}$ be two vertices on $C$.
We write $u_i C u_{i+1}$ for the subpath of $C$ with end vertices $u_i$ and $u_{i+1}$ including all the vertices on $C$ between $u_i$ and $u_{i+1}$  when traversing $C$ in clockwise direction.

Let $G$ and $H$ be two vertex-disjoint graphs and let $u\in V(G)$ and $v\in V(H)$ satisfy
$d_G(u)=d_H(v)$. For a fixed bijection $\theta:\partial_G(u)\rightarrow \partial_H(v)$, define
$(G(u)\odot H(v))_{\theta}$ to be the graph obtained from the union of $G-u$ and $H-v$ by joining,
for each edge $e\in \partial_G(u)$, the end of $e$ in $V(G)\setminus\{u\}$ to the end of
$\theta(e)$ in $V(H)\setminus\{v\}$.
When the bijection is not specified, $G(u)\odot H(v)$ denotes the family of all
graphs $(G(u)\odot H(v))_{\theta}$ obtained as $\theta$ ranges over all bijections
from $\partial_G(u)$ to $\partial_H(v)$.

Let $G$ be a graph and let $C=\partial(X)$ be an edge cut of $G$. We call $G/X$ and
$G/\overline{X}$ the \emph{$C$-contractions} of $G$. If
the contraction vertex in $G/X$ is $x$,  we also write $G/X\rightarrow x =G/X$.
For a matching covered graph $G$, an edge cut $C$ is called
\emph{separating} if both $C$-contractions of $G$ are also matching covered. Note that a tight cut
must be separating.

\begin{Lem}
   (\cite{Carvalho2004})
   \label{SM}
   A splicing of two matching covered graphs is also matching covered.
\end{Lem}
\begin{Pro}
   \label{2.4}
   (\cite{Carvalho2004})
   A splicing of two bicritical graphs is also bicritical.
\end{Pro}

The splicing operation does not necessarily preserve 3-connectivity. Thus, a splicing of two bricks need not be a brick. For a graph $G$, let $X\subseteq V(G)$ and $F\subseteq E(G)$. We say $X$ \emph{covers} $F$ if each edge of $F$ is incident with at least one vertex in $X$.

\begin{The}
   \label{2.1}
   (\cite{Carvalho2004})
   For $i=1,2$, let $G_i$ be a brick, $x_i\in V(G_i)$ and $G\in G_1(x_1)\odot G_2(x_2)$.
   Then $G$ is a brick if and only if no pair of vertices of $G$, one in $V(G_1)\setminus \{x_1\}$ and the other
   in $V(G_2)\setminus \{x_2\}$, covers the set of edges in $\partial_G(V(G_1)\setminus \{x_1\})$.
\end{The}

For planar graphs, we have the following notable result.

\begin{The}[Whitney, \cite{Whitney1932}]
   \label{Wh3}
   A simple 3-connected planar graph has a unique plane embedding, up to composition with a
    homeomorphism of the sphere.
\end{The}
A graph $G'$ is called a \emph{subdivision} of $G$ if $G'$ is obtained from $G$ by replacing some edges with internally disjoint paths. If every such path has an even number of vertices (equivalently, odd length), we call it an \emph{even subdivision}.
The following result follows directly from Theorem \ref{Wh3}.
\begin{Cor}
   \label{Bi3}
   Each subdivision of a simple 3-connected planar graph has a unique plane embedding, up to composition with a homeomorphism of the sphere.
\end{Cor}

\section{Graphs with few perfect matchings}

Let $\mathcal{M}(G)$ be the set of perfect matchings of a graph $G$.
An edge of a matchable graph $G$ is called a \emph{forcing} edge if it belongs to a unique perfect matching of $G$. For the graph with unique perfect matching, we have the following result.

\begin{The}
   (\cite{Kotzig1959})
   \label{BE}
   If a connected graph $G$ has a unique perfect matching $M$, then $M$ contains an edge $e$ such that $G-e$ is disconnected.
\end{The}


\begin{Lem}
   \label{MF}
   Let $G$ be a PM-compact brick with a multiple edge $e$.
   Then $e$ is a forcing edge of $G$.
\end{Lem}
\begin{proof}
   Let $e=uv$. Since $e$ is a multiple edge, there exists another edge $e'$, distinct from $e$, that is also incident with $u$ and $v$. Then $e$ and $e'$ form a cycle of length two.
   Since $G$ is PM-compact, $G-u-v$ contains a unique perfect matching, and thus $e$ is a forcing edge of $G$.
\end{proof}

Since all even cycles of an odd wheel contain its hub, every odd wheel contains no pair of vertex-disjoint even cycles, and thus is PM-compact.
For $W_k$ with $k\geq 5$, the forcing edges are precisely its spokes. Every edge of $W_3\cong K_4$ is forcing.
Let $G$ be a PM-compact brick. If the underlying graph of $G$ is $W_k$ with $k\geq 5$, then its multiple edges are parallel to spokes of $W_k$ by Lemma \ref{MF}.

A \emph{prism} is formed from two vertex-disjoint cycles $C_k=u_1u_2\dots u_ku_1$ and $C'_k=v_1v_2\dots v_kv_1$ by adding the edges $u_iv_i$ for $i=1,2,\dots,k$, where $k\geq 3$. When $k$ is odd, it is an \emph{odd prism}. The odd prism with $k=3$ is the \emph{triangular prism}, i.e., $\overline{C_6}$. Every prism is a simple 3-connected planar graph.

An edge $e$ of a matching covered graph $G$ is \emph{removable} if $G-e$ is matching covered.

\begin{The}[Theorem 9.14, \cite{Lucchesi2022}]
\label{NRE}
Every brick $G$ distinct from $K_4$ and $\overline{C_6}$ has at least
$\Delta(G)-2$ removable edges.
\end{The}

If $v_1$ and $v_2$ are two distinct neighbors of $u$, and $d_G(u)=2$, then the  operation of contracting $\{u,v_1,v_2\}$ to a single vertex is called the \emph{bicontraction} of $u$.
The \emph{retract} of $G$ is the graph obtained from $G$ by repeatedly bicontracting all the vertices of $G$ with degree two and with two distinct neighbors.
Let $K_2^*$ denote any graph obtained from $K_2$ by possibly adding multiple edges between its two vertices.

\begin{Lem}
\label{MG3}
(\cite{ChenDaiLu2026})
Let $G$ be a matching covered graph. If $\delta(G)\geq 3$ and $|\mathcal{M}(G)|=3$, then $G\cong K_4$ or $G$ is $K_2^*$ with exactly three edges.
\end{Lem}

\begin{Lem}
\label{MG4}
If $G$ is a brick and $|\mathcal{M}(G)|=4$, then either $G\cong \overline{C_6}$ or $G$ can be obtained from $K_4$ by adding exactly one multiple edge.
\end{Lem}
\begin{proof}
  Since $K_4$ has exactly three perfect matchings, $G\not\cong K_4$. By Theorem \ref{NRE}, either $G\cong \overline{C_6}$ or $G$ has a removable edge $e$; equivalently, $G-e$ is matching covered.
  Since $e$ belongs to at least one perfect matching of $G$, $|\mathcal{M}(G-e)|\leq 3$. Since $G-e$ is matching covered and $\Delta(G-e)\geq 3$, $|\mathcal{M}(G-e)|=3$. Let $H$ be the retract of $G-e$. Then $\delta(H)\geq 3$ and $|\mathcal{M}(H)|=|\mathcal{M}(G-e)|=3$.
  By Lemma \ref{MG3}, $H\cong K_4$ or $H$ is $K_2^*$ with exactly three edges.
  Since both possible graphs $H$ are cubic, the retraction cannot involve a bicontraction, which would produce a vertex of degree at least four. Hence $H=G-e$.
  Moreover, as $G$ is a brick, we have $|V(G)|\geq 4$. Therefore, $G\cong K_4+e$.
 \end{proof}

\begin{The}
   \label{3.1}
   (\cite{WXM2013})
   If $G$ is a bipartite matching covered graph, then $G$ is PM-compact if and only if the retract of $G$ is $K_{3,3}$ or $K_2^*$.
\end{The}

\section{A structure theorem for planar PM-compact bricks}

For two sets $A$ and $B$, the \emph{symmetric difference} of $A$ and $B$ is defined by $A\oplus B=(A\cup B)\setminus (A\cap B)$.
If $C_1$ and $C_2$ are two cycles, let $C_1\oplus C_2=E(C_1)\oplus E(C_2)$.

\begin{Lem}
   \label{GF}For $i=1,2$, let $G_i$ be a matching covered graph, $x_i\in V(G_i)$, $G\in G_1(x_1)\odot G_2(x_2)$, and $y_1y_2\in \partial_G( V(G_1)\setminus \{x_1\})$, where $y_i\in V(G_i)\setminus \{x_i\}$.
   If $G$ is PM-compact, then $x_1y_1$ is a forcing edge of $G_1$, or $x_2y_2$ is a forcing edge of $G_2$, or both.
\end{Lem}
\begin{proof}{
   Suppose that both $x_1y_1$ and $x_2y_2$ are not forcing edges  of $G_1$ and $G_2$, respectively.}
   Then for $i=1,2$, $G_i-x_i-y_i$ contains at least two different perfect matchings $M_i$ and $M'_i$.
   Let $C_i$ be a cycle in $M_i\oplus M'_i$.
   Then $G-V(C_1)-V(C_2)$ contains a perfect matching, which is a subset of $(M_1\setminus E(C_1))\cup (M_2\setminus E(C_2))\cup \{y_1y_2\}$, a contradiction.
\end{proof}

The contractions of a separating cut preserve PM-compactness, as follows.

\begin{Lem}
   \label{SC}
   Let  $C$ be a separating cut of a matching covered graph $G$.
   If $G$ is PM-compact, then  both $C$-contractions of $G$ are PM-compact.
\end{Lem}
\begin{proof}

   Let $G_1$ and $G_2$ denote the two $C$-contractions of $G$, respectively.
     Suppose, to the contrary, that $G_1$ is not PM-compact. Then $G_1$ contains two vertex-disjoint even cycles $C_1$ and $C_2$ such that $G_1-V(C_1)-V(C_2)$ contains a perfect matching, say $M_1$.
   Let $x_1$ and $x_2$ be the contraction vertices of $G_1$ and $G_2$, respectively.

   If $x_1\notin V(C_1)\cup V(C_2)$, then $x_1$ is incident with exactly one edge of $M_1$, say $e_1$.
   Since $e_1$ is incident with the contraction vertex $x_1$, it corresponds to an edge $e\in C$. In $G_2$, the edge $e$ corresponds to an edge $e_2$ incident with the contraction vertex $x_2$. As $C$ is a separating cut, $G_2$ is matching covered. It follows that $e_2$ belongs to some perfect matching of $G_2$, say $M_2$.
  Then $C_1$ and $C_2$ are two vertex-disjoint even cycles of $G$, and $(M_1\setminus \{e_1\})\cup(M_2\setminus \{e_2\})\cup\{e\}$ is a perfect matching of $G-V(C_1)-V(C_2)$, contradicting the PM-compactness of $G$.

   Now assume that $x_1\in V(C_1)\cup V(C_2)$. By symmetry, we may assume that $x_1\in V(C_1)$. Let $\{f_1,f_1'\}= \partial(\{x_1\})\cap E(C_1)$, and let $f_2$ and $f_2'$ be the two edges in $G_2$ corresponding to
   $f_1$ and $f_1'$. As $G_2$ is matching covered, there exist two perfect matchings of $G_2$, say $M_3$ and $M_4$, such that $f_2\in M_3$ and $f_2'\in M_4$. Note that $x_2\in V(f_2)\cap V(f_2')$. The vertex $x_2$ belongs to some (even) cycle in $M_3\oplus M_4$, say $C_3$. Let $C_1'=(C_1-x_1)+(C_3-x_2)+f+f'$, where $f$ and $f'$ are the edges in $C$ corresponding to $f_1$ and $f_1'$, respectively.
   Then $C_1'$ and $C_2$ are two vertex-disjoint even cycles of $G$, and $M_1\cup(M_3\setminus E(C_3))$ is a perfect matching of $G-V(C_1')-V(C_2)$, contradicting the PM-compactness of $G$.

   Similarly, the graph $G_2$ is also PM-compact.
\end{proof}

A brick is \emph{solid} if its separating cuts are trivial.
\begin{Lem}
   (\cite{Carvalho2006})
   \label{OW}
   Every simple planar solid brick is an odd wheel.
\end{Lem}
Replacing parallel edges by single edges preserves 3-connectivity and bicriticality. Thus the underlying graph of a planar solid brick is a simple planar solid brick, so Lemma \ref{OW} also applies to its underlying graph.
A separating cut $\partial(X)$ of a brick $G$ is \emph{robust} if $b(G/X) = 1$ and $b(G/\overline{X}) = 1$.
\begin{Lem}
   (\cite{Carvalho2004,He2026})
   \label{C2.9}
      Every nonsolid brick $G$ has a robust cut $\partial(X)$ such that there exists a subset $X'$ of $X$ and a subset $X''$ of $\overline{X}$ such that $G/\overline{X'}$ is a solid brick, $G/\overline{X''}$ is a brick and the graph $H$, obtained from $G$ by contracting $X'$ and $X''$ to single vertices $x'$ and $x''$, respectively, is bipartite and matching covered, where $x'$ and $x''$ lie in different color classes of $H$.
   \end{Lem}

\begin{The}[The structure theorem]
\label{MKCo}
Every planar PM-compact nonsolid brick $G$ has a cut $C$ such that the underlying graph of one $C$-contraction is an odd wheel, and the other $C$-contraction is a planar PM-compact brick. If the underlying wheel is $W_k$ with $k\geq 5$, then all multiple edges of the former contraction are parallel to spokes.
\end{The}
\begin{proof}
   Since $G$ is a nonsolid brick, $G$ has two disjoint nonempty vertex subsets $X'$ and $X''$ such that
   $G_1=G/\overline{X'}\rightarrow \overline{x'}$, $G_2=G/\overline{X''}\rightarrow \overline{x''}$ and $H=(G/X'\rightarrow x')/X''\rightarrow x''$, where $G_1$ is a solid brick, $G_2$ is a brick, and $H$ is bipartite and matching covered  by Lemma \ref{C2.9}.
   Since $G$ is planar, so are $G_1$, $G_2$ and $H$. By Lemma \ref{OW} and the preceding observation, the underlying graph of $G_1$ is an odd wheel.


    Since $G/X'$ is a splicing of two matching covered graphs $G_2$ and $H$, $G/X'$ is matching covered by Lemma \ref{SM}, and thus $\partial_G(X')$ is a separating cut of $G$.
   By Lemma \ref{SC}, both $G_1$ and $G/X'$ are PM-compact. Since $\partial_{G/X'}(X'')$ is a separating cut of $G/X'$, both $G_2$ and $H$ are PM-compact by Lemma \ref{SC}.
   By Theorem \ref{3.1}, the retract of $H$ is $K_{3,3}$ or $K_2^*$. Since $G$ is 3-connected, $\delta(H)\geq 3$ and thus the retract of $H$ is itself.
   Since $H$ is planar, $H$ is $K_2^*$, and thus $G_2=G/X'$.
   Since $G_2$ is a planar PM-compact brick, $\partial_G(X')$ is the robust cut we need.
   If the underlying graph of $G_1$ is $W_k$ with $k\geq 5$, then all its multiple edges are parallel to spokes by Lemma \ref{MF}, since no rim edge is forcing.
\end{proof}


\section{\label{Sec3}Planar PM-compact splicings}

This section presents two fundamental graphs $G^*$ and $G^{\triangle}$ (see Fig. \ref{f2-0}) needed for the proof of Theorem \ref{MaT}.
Throughout this section, whenever a figure or one of its subfigures is cited, the vertex labels are those shown in the cited figure.
Let $H$ be a subgraph of a graph $G$. $H$ is called a \emph{spanning subgraph} of $G$ if $V(H)=V(G)$.
\begin{figure}[h]
   \centering
   \includegraphics[scale=0.3]{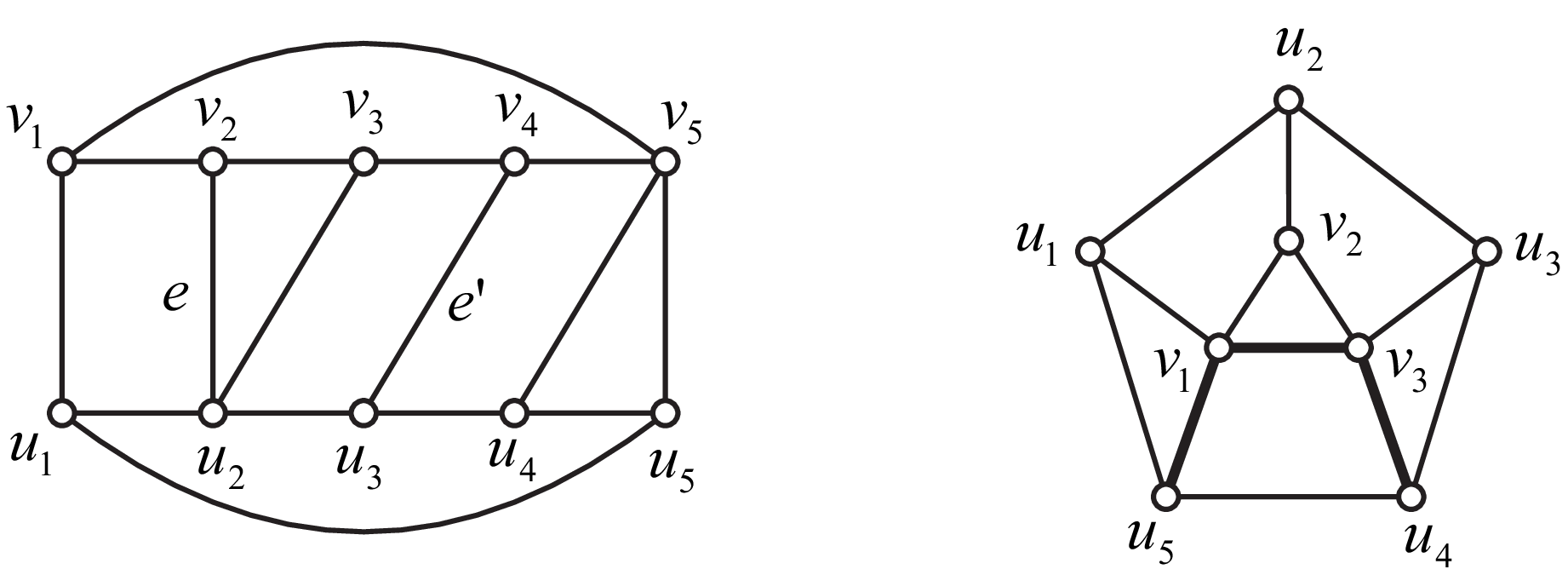}
   \caption{\label{f2-0}Left: the graph $G^*$; Right: the graph $G^{\triangle}$.}
\end{figure}
\subsection{A splicing of two large odd wheels}

A cycle $C$ of a graph $G$ is called a \emph{Hamilton cycle} of $G$ if $V(C)=V(G)$.
The next lemma treats the case in which two graphs obtained from odd wheels by allowing parallel spokes are spliced at their hubs. Lemma \ref{W5W5} will then show that every planar PM-compact splicing of two such graphs is necessarily a hub-to-hub splicing, and hence follows from this lemma.

\begin{Lem}
   \label{W5W5h}
   Let $G$ be obtained by splicing at their hubs two graphs, each formed from an odd wheel with at least six vertices by allowing parallel spokes. If $G$ is a planar PM-compact brick, then $G$ contains a spanning subgraph isomorphic to $G^*$.
\end{Lem}
\begin{proof}
   Note that taking the underlying simple graph preserves the hypotheses and the existence of a spanning subgraph $G^*$. Thus we may assume that $G$ is simple.
   Write $G\in G_1(h_1)\odot G_2(h_2)$, where, for $i=1,2$, $G_i$ is obtained from the odd wheel $W_{k_i}$ by allowing parallel spokes, $k_i\geq 5$, and $h_i$ and $R_i$ are its hub and rim, respectively.
   Recall that $E[R_1,R_2]$ denotes the set of edges of $G$ with one end in $V(R_1)$ and the other in $V(R_2)$.
   \begin{Cla}
   \label{nc1}
   Each perfect matching of $G$ contains at most three edges in $E[R_1,R_2]$.
   \end{Cla}
   \begin{proof}
   Suppose that $G$ contains a perfect matching $M$ with $|M\cap E[R_1,R_2]|\geq 4$.
   Since $|V(R_1)|$ is odd, $|M\cap E[R_1,R_2]|$ is an odd integer.
   Let $G'=G[E(R_1)\cup E(R_2)\cup M]$.
   Since $G$ is PM-compact and $G'$ is a spanning subgraph of $G$, $G'$ is PM-compact by Prop. \ref{PM-P1}. Since $G'$ is a subdivision of an odd prism, Corollary \ref{Bi3} implies that $G'$ has a unique plane embedding.
   List the edges of $M\cap E[R_1,R_2]$ in their cyclic order along $R_1$. Every pair of successive edges in this list lies on an $M$-alternating facial cycle of $G'$. Consequently, $G'$ contains two vertex-disjoint $M$-alternating cycles, contradicting Prop. \ref{PM-P2}.
   \end{proof}
   Since $G$ is a brick, the nontrivial edge cut $E[R_1,R_2]$ is not a tight cut of $G$ and thus $G$ contains a perfect matching $M$ such that $|M\cap E[R_1,R_2]|\geq 3$. By Claim \ref{nc1},  $M\cap E[R_1,R_2]$ has exactly three edges $u_1v_1$, $u_2v_2$ and $u_3v_3$.
    Let $P_i=u_iR_1u_{i+1}$ and $P_i'=v_iR_2v_{i+1}$ for $1\leq i\leq 3$, where the subscripts are taken modulo 3. Since all internal vertices of $P_i$ and $P_i'$ are matched by $M\cap (E(P_i)\cup E(P_i'))$, each of $P_i$ and $P_i'$ has an odd number of edges.
   \begin{figure}[h]
      \centering
      \includegraphics[scale=0.24]{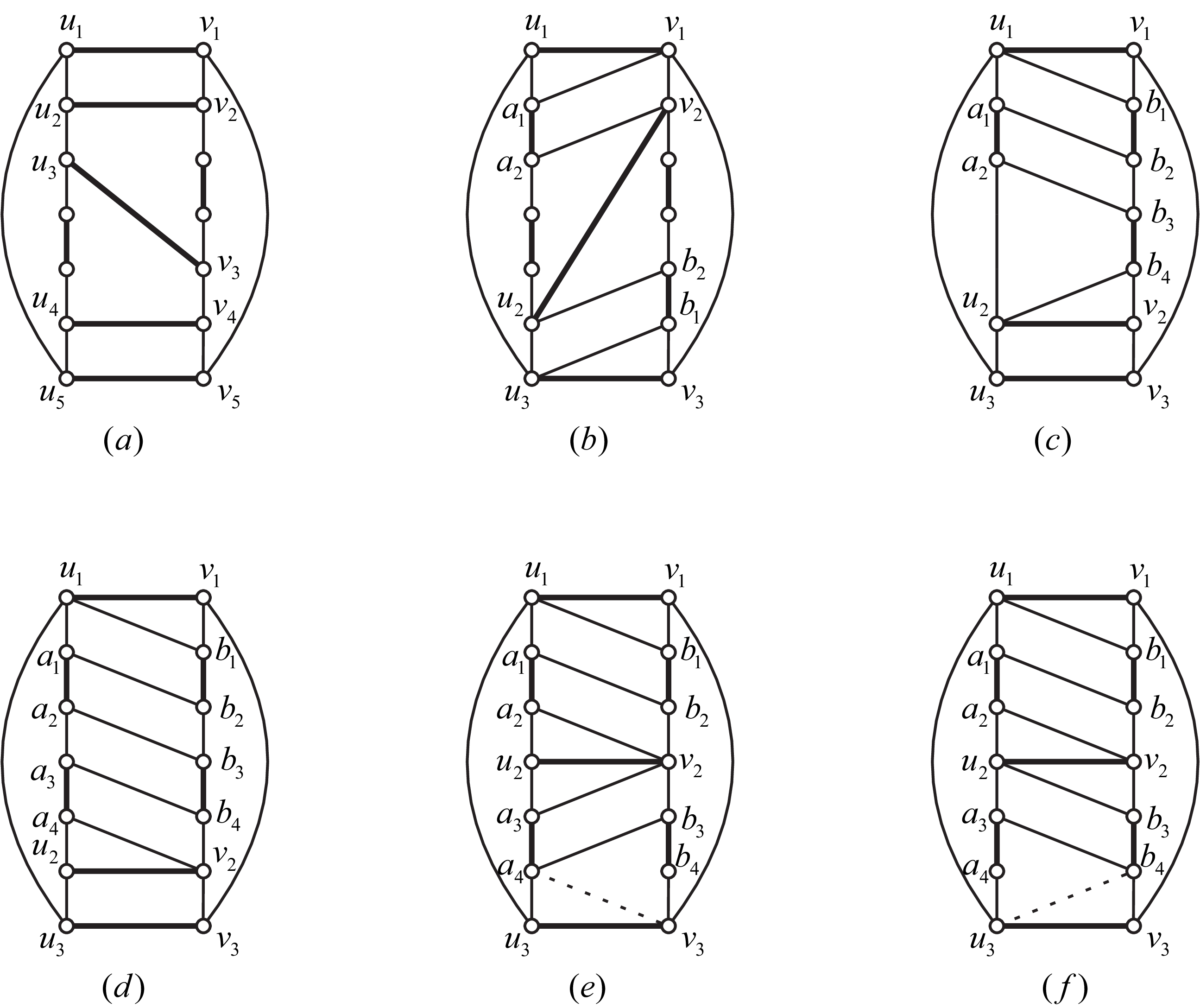}
      \caption{\label{f2-3-1}Illustration for the proof of Lemma \ref{W5W5h}.}
   \end{figure}
\begin{Cla}
   \label{new-cla1}
 $|V(P_i)|\geq 4$ and $|V(P_i')|\geq 4$ for some $i\in \{1,2,3\}$.
\end{Cla}
\begin{proof}
   Suppose that for every $i\in\{1,2,3\}$, either $|V(P_i)|=2$ or $|V(P_i')|=2$. Since $|V(R_i)|\geq 5$ for $i=1,2$, at least one of $P_1$, $P_2$ and $P_3$ has at least four vertices, and similarly for $P_1'$, $P_2'$ and $P_3'$.
   By symmetry, we may relabel the paths so that $|V(P_1)|\geq 4$ and $|V(P_2')|\geq 4$. Consequently, $|V(P_1')|=|V(P_2)|=2$, as illustrated in Fig. \ref{f2-3-1} $(b)$.
   Let $P_1=u_1a_1a_2\dots a_{2p}u_2$ and $P_2'=v_2b_{2q}\dots b_2b_1v_3$, where $p$ and $q$ are positive integers.
   Since $P_1'=v_1v_2$ and $P_2=u_2u_3$, $N_G(a_i)\cap V(R_2)\subseteq \{v_1,v_2\}$ and $N_G(b_i)\cap V(R_1)\subseteq \{u_2,u_3\}$ by the planarity of $G$. If $a_1v_2\in E(G)$, then $a_1v_2u_2a_{2p}a_{2p-1}\dots a_1$ and $u_1v_1+P_3'+v_3u_3+P_3$ are two vertex-disjoint $M$-alternating cycles, a contradiction to Prop. \ref{PM-P2}. So $a_1v_1\in E(G)$. If $a_2v_1\in E(G)$, then $a_2v_1u_1a_1a_2$ and $u_2v_2+P_2'+v_3u_3u_2$ are
   two vertex-disjoint $M$-alternating cycles, a contradiction to Prop. \ref{PM-P2}. So $a_2v_2\in E(G)$. By symmetry,
   $\{b_{1}u_3,b_{2}u_2\}\subseteq E(G)$. Then $u_2b_2b_3\dots b_{2q}v_2a_2a_3\dots a_{2p}u_2$ and $u_1a_1v_1+P_3'+v_3b_1u_3+P_3$ are two vertex-disjoint even cycles containing all the vertices of $G$, a contradiction.
\end{proof}
  By Claim \ref{new-cla1},
  we can assume that $|V(P_1)|\geq 4$ and $|V(P_1')|\geq 4$.
 Label the internal vertices so that $P_1=u_1a_1a_2\dots a_{2p}u_2$ and $P_1'=v_1b_1b_2\dots b_{2q}v_2$, where $p$ and $q$ are positive integers.
  Let $u_1=a_0$, $u_2=a_{2p+1}$, $v_1=b_0$ and $v_2=b_{2q+1}$.
  \begin{Cla}
  \label{nnn4}
   If $a_ib_j\in E(G)\setminus \{a_0b_0,a_{2p+1}b_{2q+1}\}$ ($0\leq i\leq 2p+1$, $0\leq j\leq 2q+1$), then $i-j\equiv 1$ (mod 2). For the graph $G-u_1v_1-u_2v_2$, each vertex of $P_1$ (resp. $P_1'$) has at most one neighbor in $V(P_1')$ (resp. $V(P_1)$).
  \end{Cla}
  \begin{proof}
  Suppose that $i-j\equiv 0$ (mod 2). If $i$ and $j$ are odd, then, as in the argument for $a_1v_2\in E(G)$ in Claim \ref{new-cla1}, $\{a_ib_j,u_2v_2\}$ and $\{u_1v_1,u_3v_3\}$ belong to two vertex-disjoint $M$-alternating cycles, respectively, a contradiction to Prop. \ref{PM-P2}. If $i$ and $j$ are even, then, as in the argument for $a_2v_1\in E(G)$ in Claim \ref{new-cla1}, $\{a_ib_j,u_1v_1\}$ and $\{u_2v_2,u_3v_3\}$ belong to two vertex-disjoint $M$-alternating cycles, respectively, again a contradiction. So $i-j\equiv 1$ (mod 2).

  Next we only consider the edges in $G-u_1v_1-u_2v_2$. If $P_1$ contains a vertex $a_{i'}$ with two different neighbors $b_{i_1}$ and $b_{i_2}$ in $V(P_1')$, then every vertex on the subpath $b_{i_1}P_1'b_{i_2}$ has a neighbor on $R_1$; planarity forces that neighbor to be $a_{i'}$.
  Further, $b_{i_1}P_1'b_{i_2}$ contains a vertex $b_{k}$ with $k-i'\equiv 0$ (mod 2), a contradiction. So each vertex of $P_1$ has at most one neighbor in $V(P_1')$. By symmetry,
  each vertex of $P_1'$ has at most one neighbor in $V(P_1)$.
  \end{proof}

\begin{Cla}
   \label{new-cla3}
   $|V(P_1)|=|V(P_1')|=4$.
\end{Cla}
\begin{proof}
Recall that $|V(P_1)|\geq 4$ and $|V(P_1')|\geq 4$.
 Suppose that $|V(P_1)|>4$ or $|V(P_1')|>4$. By symmetry, we may assume that $|V(P_1')|>4$. Since $P_1'$ is an odd path, $|V(P_1')|\geq 6$. If $|V(P_1')|>6$, then $P_1'$ has a subpath with five inner vertices; let $S$ be the vertex set of these five vertices. By $\delta(G)\geq 3$, planarity, and Claim \ref{nnn4}, $E[S,V(P_1)]$ is a matching $M'$ of size five whose ends on $P_1$ occur in the same cyclic order and alternate in index parity. Deleting $V(M')$ from the two odd rims leaves only even-order paths, so $M'$ extends to a perfect matching of $G$, contradicting Claim \ref{nc1}. Hence $|V(P_1')|=6$.
If $|V(P_1)|>4$, then $|V(P_1)|=6$ by symmetry. By Claim \ref{nnn4}, we can also find a perfect matching of $G$ that contains five edges of $E[P_1,P_1']$ (see Fig. \ref{f2-3-1} ($d$)), contradicting Claim \ref{nc1}. So $|V(P_1)|=4$.
By Claim \ref{nnn4}, $a_ib_{i+1}\in E(G)$ for $0\leq i\leq 3$. As shown in Fig. \ref{f2-3-1} $(c)$,  $a_1a_2b_3b_2a_1$ and $u_1b_1v_1+P_3'+P_2'+v_2b_4u_2+P_2+P_3$ are two vertex-disjoint even cycles containing all the vertices of $G$, a contradiction.
\end{proof}

\begin{Cla}
   \label{new-cla4}
   For $i=2,3$, $|V(P_i)|=|V(P_i')|=2$.
\end{Cla}

\begin{proof}
By symmetry, it suffices to show that $|V(P_2)|=|V(P_2')|=2$.
 By Claims \ref{nnn4} and \ref{new-cla3}, $E[P_1,P_1']\setminus \{u_1v_1,u_2v_2\}$ is a matching $M_1$ of size three, and $M_1=\{a_0b_1,a_1b_2,a_2b_3\}$ or $M_1=\{a_1b_0,a_2b_1,a_3b_2\}$. First we consider the case that $M_1=\{a_0b_1,a_1b_2,a_2b_3\}$.
 If both $P_2$ and $P_2'$ have at least four vertices, then by the argument of Claim \ref{new-cla3}, $|V(P_2)|=|V(P_2')|=4$ and $E[P_2,P_2']\setminus \{u_2v_2,u_3v_3\}$ is a matching of size three, say $M_2$.
 If one of $P_2$ and $P_2'$ has at least four vertices and the other has exactly two vertices, then $E[P_2,P_2']\setminus \{u_2v_2,u_3v_3\}$ is a matching of size two by the argument of Claim \ref{nnn4}, say $M_2'$.

 If $M_1\cup M_2$ (resp. $M_1\cup M_2'$) is a matching, then it contains five edges of a perfect matching of $G$, a contradiction to Claim \ref{nc1}. Otherwise, $a_3v_2\in E(G)$. Deleting the vertex sets of the two vertex-disjoint even cycles $u_1b_1b_2a_1u_1$ and $a_2v_2a_3u_2a_2$ leaves a subgraph with a Hamilton cycle, a contradiction.
 In this case, $|V(P_2)|=|V(P_2')|=2$. For $M_1=\{a_1b_0,a_2b_1,a_3b_2\}$, the preceding argument shows that $|V(P_2)|=|V(P_2')|=2$.
\end{proof}
By Claims \ref{new-cla3} and \ref{new-cla4}, $k_1=k_2=5$, so the underlying graph of each $G_i$ is $W_5$. Recall that $E[P_1,P_1']\setminus \{u_1v_1,u_2v_2\}$ is a matching $\{a_0b_1,a_1b_2,a_2b_3\}$ or $\{a_1b_0,a_2b_1,a_3b_2\}$ (in Claim \ref{new-cla4}), and thus $G^*$ is a spanning subgraph of $G$.
\end{proof}

\begin{Lem}
   \label{W5W5} Let $G$ be a splicing of two graphs, each formed from an odd wheel with at least six vertices by allowing parallel spokes.
    If $G$ is a planar PM-compact brick, then $G$ contains a spanning subgraph isomorphic to $G^*$.
\end{Lem}
\begin{proof} Let $G$ be a splicing of $G_1$ and $G_2$ at $x_1$ and $x_2$, respectively. For $i=1,2$, let $G_i$ be obtained from the odd wheel $W_{k_i}$ by allowing parallel spokes, where $k_i\geq 5$, and let $h_i$ and $R_i$ be its hub and rim, respectively.
   By Lemma \ref{W5W5h}, it suffices to show that $x_1$ and $x_2$ are hubs of $G_1$ and $G_2$,
   respectively. Let $X_i=V(G_i)\setminus \{x_i\}$ for $i=1,2$. Since $G$ is a brick, the nontrivial edge cut $E_G[X_1,X_2]$ is not a tight cut of $G$. So $G$ contains a perfect matching $M$ such that $|M\cap E_G[X_1,X_2]|\geq 3$.

   Suppose that one of $x_1$ and $x_2$ is not the hub of its corresponding wheel, say $x_2\neq h_2$.
   Then $x_2$ has exactly three distinct neighbors in $G_2$, which cover the edges in $\partial_G(X_2)$. Thus $M\cap E_G[X_1,X_2]$ contains exactly three edges $e_1$, $e_2$ and $e_3$.
   If $x_1\neq h_1$, then we can assume that $N_{G_i}(x_i)=\{h_i,y_i,z_i\}$.
   Then at least one of $\{e_1,e_2,e_3\}$ is not incident with any vertex in $\{h_1,h_2\}$; let $e_1=y_1y_2$.
   It follows that $y_1x_1$ is not a forcing edge of $G_1$. At the same time,
   $y_2x_2$ is not a forcing edge of $G_2$, a contradiction to Lemma \ref{GF}.
  So $x_1=h_1$.

   \begin{figure}[h]
      \centering
      \includegraphics[scale=0.25]{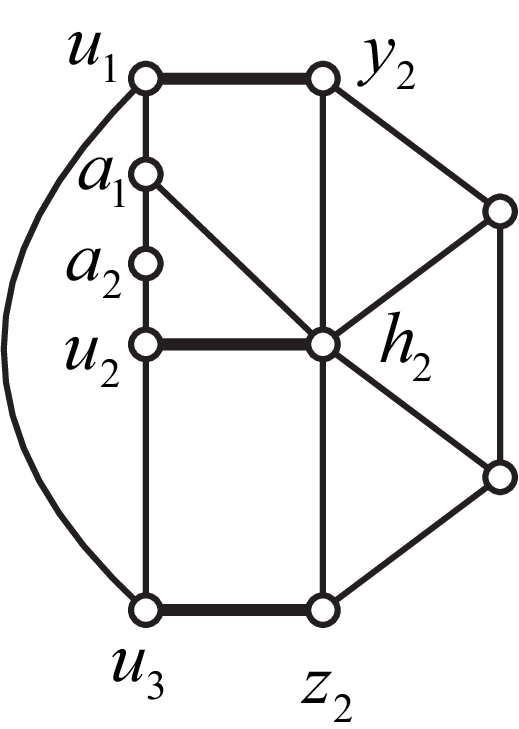}
      \caption{\label{f2-2-1}Illustration for the proof of Lemma \ref{W5W5}.}
   \end{figure}

   Let $R_1$ be the rim of $G_1$. Then $X_1=V(R_1)$.
   Recall that $N_{G_2}(x_2)=\{h_2,y_2,z_2\}$.
   Relabel the edges and vertices so that $e_1=u_1y_2$, $e_2=u_2h_2$ and $e_3=u_3z_2$, where $\{u_1,u_2,u_3\}\subseteq V(R_1)$.
   Index $u_1,u_2,u_3$ in their cyclic order on $R_1$. They divide $R_1$ into three edge-disjoint paths $P_1,P_2,P_3$. Since $e_1,e_2,e_3\in M$, each $P_i$ has an odd number of edges.
   Since $|V(R_1)|=k_1\geq 5$, at least one of $\{P_1,P_2,P_3\}$ has at least three edges, say $P_1$.
   Since $G$ is PM-compact and $E_G[X_1,X_2]$ is a separating cut of $G$, $G_2$ is PM-compact by Lemma \ref{SC}.
   Since neither $x_2y_2$ nor $x_2z_2$ is a forcing edge of $G_2$, neither $x_2y_2$ nor $x_2z_2$ is a multiple edge of $G_2$ by Lemma \ref{MF}.
   Thus, each of $y_2$ and $z_2$ has a unique neighbor in $V(R_1)$, i.e., $u_1$ and $u_3$ respectively. Then for each internal vertex $x$ of $P_1$, $xh_2\in E(G)$. Since $u_2h_2\in E(G)$, $u_2$ must be an end vertex of $P_1$ by the planarity of $G$. Let $P_1=u_1a_1a_2\dots a_{2p}u_2$.
  Since $a_1h_2\in E(G)$, $u_2h_2\in M$ and all $a_i$'s are matched by $M\cap E(P_1)$, $a_1a_2\dots a_{2p}u_2h_2a_1$ is an $M$-alternating cycle, denoted by $C_1$. As shown in Fig. \ref{f2-2-1}, $u_1y_2$ and $u_3z_2$ belong to the same facial cycle $C_2$ (exterior facial cycle) of $G$ and $C_2$ is an $M$-alternating cycle vertex-disjoint from $C_1$, a contradiction to Prop. \ref{PM-P2}.
   So $x_1=h_1$ and $x_2=h_2$. By Lemma \ref{W5W5h}, $G$ contains a spanning subgraph isomorphic to $G^*$.
\end{proof}

\begin{Lem}\label{GstarChain}
   If a planar PM-compact brick $G$ contains a spanning subgraph isomorphic to $G^*$, then $G$ admits the chain in Theorem \ref{MaT}.
\end{Lem}
\begin{figure}[ht]
   \centering
   \includegraphics[scale=0.28]{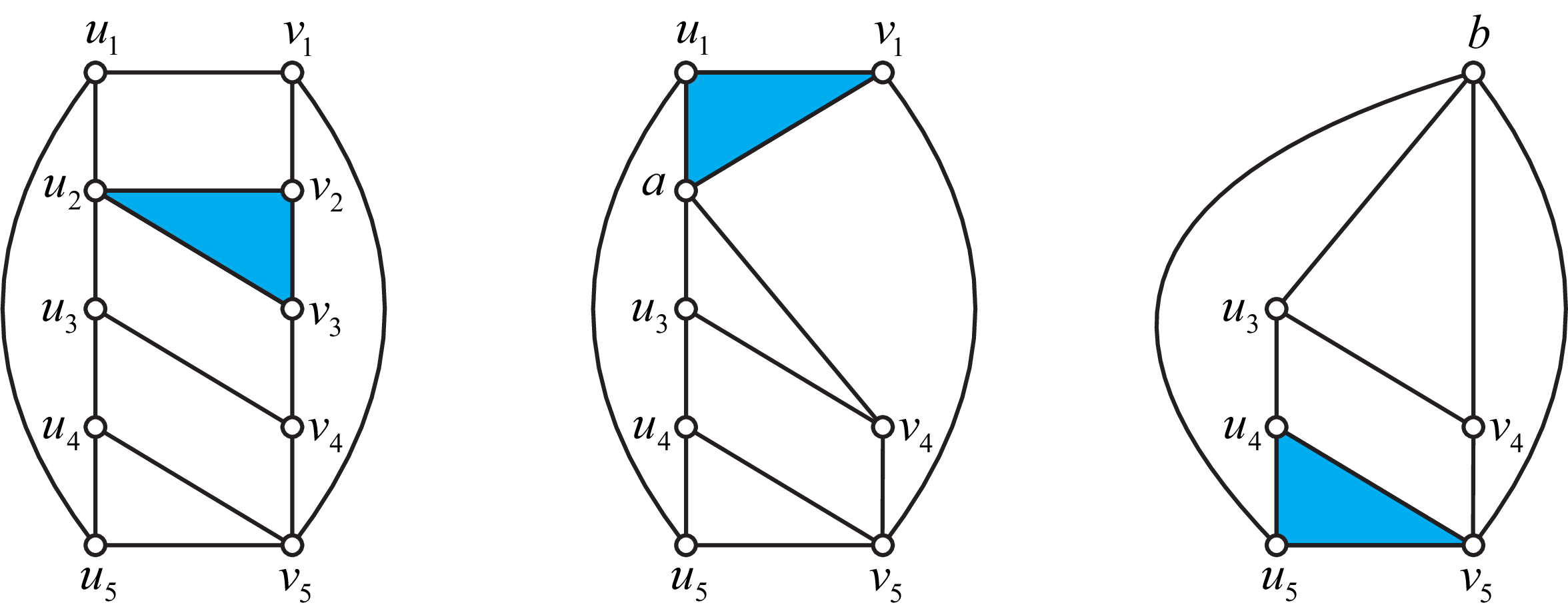}
   \caption{A spanning subgraph of $G_0$ isomorphic to $G^*$ (left) and its images in $G_1$ (middle) and $G_2$ (right). The shaded triangles are contracted.\label{fig:GstarChain}}
\end{figure}
\begin{proof}
   Identify a spanning subgraph of $G$ isomorphic to $G^*$ with the left picture in Fig. \ref{fig:GstarChain}. Set $G_0=G$, $G_1=G_0/\{u_2,v_2,v_3\}\rightarrow a$, $G_2=G_1/\{a,u_1,v_1\}\rightarrow b$, and $G_3=G_2/\{u_4,u_5,v_5\}\rightarrow c$.
   The underlying graph of $G_3$ is $K_4$, so $G_3$ is a brick. For each of the three triangle cuts, the opposite contraction has underlying graph $K_4$, and the cut contains a matching of size three (see Fig. \ref{fig:GstarChain}). Applying Theorem \ref{2.1} from $G_3$ shows that $G_2$ and $G_1$ are bricks.
   The three cuts are therefore separating, and Lemma \ref{SC} gives the PM-compactness of $G_1,G_2,G_3$. Triangle contractions preserve planarity, and $K_4$ is an odd wheel. Thus $G_0,G_1,G_2,G_3$ is the required chain.
\end{proof}

\subsection{A large odd wheel and \texorpdfstring{$K_4$}{K4}}
Lemma \ref{W5W5} identifies a spanning subgraph isomorphic to $G^*$ in every planar PM-compact splicing of two graphs formed from odd wheels with at least six vertices by allowing parallel spokes.
Next we consider a splicing of a wheel with possible parallel spokes and a graph whose underlying graph is $K_4$.

\begin{Lem}
   \label{WkPath}
   Let $G_1$ be obtained from $W_k$ ($k\geq 5$) by allowing parallel spokes, with hub $h$, and let $G_2$ have underlying graph $K_4$. If $G\in G_1(h)\odot G_2(v)$ is a planar PM-compact brick, then $k=5$ and $G$ has a spanning subgraph isomorphic to $G^{\triangle}$ (see Fig. \ref{f2-0}).
\end{Lem}
\begin{proof}
   Let $R=u_1u_2\dots u_ku_1$ be the rim of $W_k$, and let $T$ be a triangle in $G-V(R)$, with $V(T)=\{v_1,v_2,v_3\}$.
   Since $G$ is a brick, every nontrivial cut is not a tight cut.
   Since $E[T,R]$ is a nontrivial cut, it is not a tight cut, i.e.
   $G$ has a perfect matching $M$ such that $|M\cap E[T,R]|\geq 3$. Since $|V(T)|=3$, $|M\cap E[T,R]|= 3$.
   Let $H$ be the underlying graph of $G$, chosen so that $M\cup E(T)\subseteq E(H)$. Each vertex of $R$ has a neighbor in $T$, so we can obtain $G'$ from $H$ by deleting edges of $E_H[T,R]\setminus M$ until $d_{G'}(u_i)=3$ for $1\leq i\leq k$. Since $G'$ is a spanning subgraph of $G$, it is PM-compact by Prop. \ref{PM-P1}.
   Since $G'[M\cup E(T)\cup E(R)]$ is a subdivision of an odd prism, Corollary \ref{Bi3} implies that this subgraph has a unique plane embedding. For each $v_i$ ($i=1,2,3$), the vertices of $N_{G'}(v_i)\cap V(R)$ induce a path of $R$. Let $S_i=N_{G'}(v_i)\cap V(R)$ for $i=1,2,3$. Since $M\subseteq E(G')$, $|S_i|\geq 1$. Since $d_{G'}(u_i)=3$, each pair of $S_1$, $S_2$ and $S_3$ is disjoint and $\displaystyle\sum_{i=1}^3|S_i|=|V(R)|=k$ is odd.

   If at least two of $|S_1|$, $|S_2|$ and $|S_3|$ are at least three, say $|S_1|$ and  $|S_2|$, then for $i=1,2$,
   there is a cycle of length four $C_i$ in $G'[\{v_i\}\cup S_i]$, which contains a vertex adjacent to some vertex in $S_3$.
   If $|S_3|$ is odd, then $G'[\{v_3\}\cup S_3]$ has a perfect matching, say $M_3$. Since $|V(C_i)\cap V(R)|=3$, $G'-V(C_1)-V(C_2)-V(M_3)$ is either empty or an odd subpath of $R$.
   So $G'-V(C_1)-V(C_2)$ has a perfect matching, a contradiction to the PM-compactness of $G'$. If $|S_3|$ is even, then either $|S_1|$ or $|S_2|$ is even, say the former. It implies that $|S_1|\geq 4$. Let $s\in S_1\cap N_{G'}(S_3)$. Then
   $G'[\{v_3\}\cup S_3\cup \{s\}]$ has a perfect matching, say $M_3'$.
   In this case, let $C_1'$ be the cycle of length four in $G'[\{v_1\}\cup (S_1\setminus \{s\})]$, which contains a vertex adjacent to $s$. Then $G'-V(C_1')-V(C_2)-V(M_3')$ is either empty or an odd subpath of $R$. So $G'-V(C_1')-V(C_2)$ has a perfect matching, a contradiction.

   If exactly one of $|S_1|$, $|S_2|$ and $|S_3|$ is at least three, say $|S_1|$, then for $i=2,3$, $|S_i|\leq 2$. If $|S_1|$ is odd, then the Hamilton cycle $C_h$ of $G'[\{v_1\}\cup S_1]$ is an even cycle. Since $G'-V(C_h)$ has a Hamilton cycle, a contradiction. If $|S_1|$ is even, then either $|S_2|$ or $|S_3|$ is odd, say the former. It implies that $|S_2|=1$. Let $C_1$ be the cycle of length four in $G'[\{v_1\}\cup S_1]$, which has a vertex adjacent to the unique vertex in $S_2$, and $C_2$ be the cycle of length four containing both $v_2$ and $v_3$. Then $G'-V(C_1)-V(C_2)$ is empty or an odd subpath of $R$, thus it has a perfect matching, a contradiction.

    By the preceding discussion, $1\leq |S_i|\leq 2$ for $i=1,2,3$. Since $\displaystyle\sum_{i=1}^3|S_i|$ is odd, one of $|S_1|$, $|S_2|$ and $|S_3|$ equals 1 and the other two equal 2. Thus $G^{\triangle}$ is a spanning subgraph of $G'$, and hence of $G$.
\end{proof}
\section{Proof of Theorem \ref{MaT}}

\noindent\textbf{Theorem \ref{MaT}} \textit{Let $G$ be a planar PM-compact brick. Then there exists a sequence of planar PM-compact bricks $G_0, G_1, \dots, G_n$ such that $G_0=G$, the underlying graph of $G_n$ is an odd wheel, and each $G_i$ {\rm (}$0\le i<n${\rm )} has a triangle $T_i$ with $G_i/T_i=G_{i+1}$.}
\begin{proof}
   Apply induction on $|V(G)|$.
   If $G$ is solid, then its underlying graph is an odd wheel by Lemma \ref{OW}, and we are done.
   If $G$ is nonsolid, then $|V(G)|\geq 6$.
   By Theorem \ref{MKCo}, $G$ has a robust cut $\partial(X)$ such that the underlying graph of $G/\overline{X}$ is an odd wheel,
   and $G/X$ is a planar PM-compact brick.
   If $|V(G)|=6$, then the underlying graph of $G[X]$ is isomorphic to a triangle. Hence the result follows.
   Next assume that $|V(G)|\geq 8$. If $|V(G/X)|=4$ or $|V(G/\overline{X})|=4$, then we are done.
   If the underlying graphs of both $G/X$ and $G/\overline{X}$ are odd wheels with at least six vertices, then $G$ contains a spanning subgraph isomorphic to $G^*$ by Lemma \ref{W5W5}. Lemma \ref{GstarChain} gives the required chain.
   It remains to consider the case where the underlying graph of $G/\overline{X}\rightarrow \overline{x}$ is an odd wheel with at least six vertices and the underlying graph of $G/X\rightarrow x$ is not an odd wheel. Since $G/X$ is a planar PM-compact brick, $G/X$ is nonsolid by Lemma \ref{OW}.
   By the induction hypothesis, $G/X$ contains a triangle $T$ such that $(G/X)/T\rightarrow t$ is a brick. Let $V(T)=\{t_1,t_2,t_3\}$ and $G'=(G/X)/T$. Then $|V(G')|\geq 4$. Let $Y=V(G)\setminus (X\cup V(T))$. Then $|Y|\geq 2$. We have the following claims, the proofs of which will be provided later.
   \setcounter{Cla}{0}
   \begin{Cla}
      \label{cla8}
      If $x\notin \{t_1,t_2,t_3\}$ in the graph $G/X$, then $G/T$ is a brick.
   \end{Cla}
      \begin{Cla}
      \label{cla7}
      If $x\in \{t_1,t_2,t_3\}$ in the graph $G/X$, then either $G[X\cup V(T)]$ contains a triangle $T_x$ such that $G/T_x$ is a brick, or $G[Y]$ is a triangle such that $G/Y$ is a brick.
   \end{Cla}
   By Claims \ref{cla8} and \ref{cla7}, there is a triangle $S$ such that $G/S$ is a brick. Since $G$ is a brick, the underlying graph of the other contraction of $\partial_G(V(S))$ is $K_4$. Thus $\partial_G(V(S))$ is a separating cut, and Lemma \ref{SC} makes $G/S$ PM-compact. Triangle contraction preserves planarity and reduces the number of vertices, so induction completes the chain.
\end{proof}

   We now prove Claims \ref{cla8} and \ref{cla7}.
   Since $G/X$ is a brick, every tight cut of $G/X$ is trivial.
   Since $\partial_{G/X}(V(T))$ is not a trivial edge cut of $G/X$,
   it is not a tight cut of $G/X$. So $G/X$ contains a perfect matching $M_{G/X}^T$ such that $|M_{G/X}^T\cap \partial_{G/X}(V(T))|\geq 3$. As $|V(T)|=3$,
   \begin{equation}
   \label{eq1}
       |M_{G/X}^T\cap \partial_{G/X}(V(T))|=3.
   \end{equation}

   Similarly, by the fact that $G$ is a brick, $G$ contains a perfect matching $M_G^X$ such that

   \begin{equation}
   \label{eq2}
       |M_G^X\cap \partial_G(X)|\geq 3.
   \end{equation}

\subsection{Proof of Claim \ref{cla8}}
Suppose to the contrary that $G/T$ is not a brick.
      As $X\cap V(T)=\emptyset$, $V(T)\subseteq \overline{X}$. Note that
      $G/T$ is a splicing of $(G/T)/X$ and $G/\overline{X}$. Since $(G/T)/X=(G/X)/T=G'$, $(G/T)/X$ is a brick. Since $G/\overline{X}$ is a brick whose underlying graph is an odd wheel,
      $G/T$ is a splicing of two bricks $G'$ and $G/\overline{X}$. From the assumption that $G/T$ is not a brick, we get that $E_{G/T}[X, \{t\}\cup Y]$ is covered by a vertex $s$ in $X$ and a vertex $a$ in $\{t\}\cup Y$ by Theorem \ref{2.1}. Since $|X|\geq 5$ and $|\{t\}\cup Y|\geq 3$, $\{s,a\}$ is a 2-vertex cut of $G/T$. Since $G$ is 3-connected but $G/T$ is not, $t\in \{s,a\}$. Since $s\in X$ and $a \in \{t\}\cup Y$, $a=t$. So
      \begin{equation}
          \label{eq3}
          E_G[X,Y]=E_G[\{s\},Y].
      \end{equation}

      It implies that $|M_G^X\cap E_G[X,Y]|\leq 1$.
   Since $|V(T)|=3$, $|M_G^X\cap E_G[X,V(T)]|\leq 3$.
   Thus, $|M_G^X\cap E_G[X,V(T)\cup Y]|\leq 4$.
   Since $|X|$ is odd, $|M_G^X\cap \partial_G(X)|$ is odd.
   Since $|M_G^X\cap E_G[X,V(T)\cup Y]|=
   |M_G^X\cap \partial_G(X)|\geq 3$,
   \begin{equation}
       \label{eq4}
       |M_G^X\cap \partial_G(X)|=3,
   \end{equation}
   and $|M_G^X\cap E_G[X,V(T)]|\geq 2$.
   Let $M_G^X\cap \partial_G(X)=\{e_1,e_2,e_3\}$ and $V(e_i)\cap X=\{x_i\}$ for $i=1,2,3$.
   Relabel the vertices of $T$ so that $V(e_i)\cap \overline{X}\subseteq \{t_i\}\cup Y$ for $i=1,2,3$.

   Note that $G[X]$ contains an odd Hamilton cycle $C$.
   Let $P_i=x_iCx_{i+1}$ for $1\leq i\leq 3$, where the subscripts are taken modulo 3. Then each $P_i$ is an $M_G^X$-alternating odd path, and the edges of $P_i$ incident with its end vertices are not in $M_G^X$.
   As $x\notin V(T)$, $x$ and $t$ are two different vertices of $G'$.
   Since $G[Y]=G'-x-t$ and $G'$ is 3-connected and bicritical, $G[Y]$ is connected and contains a perfect matching, say $M_Y$.

   \begin{figure}[h]
         \centering
         \includegraphics[scale=0.3]{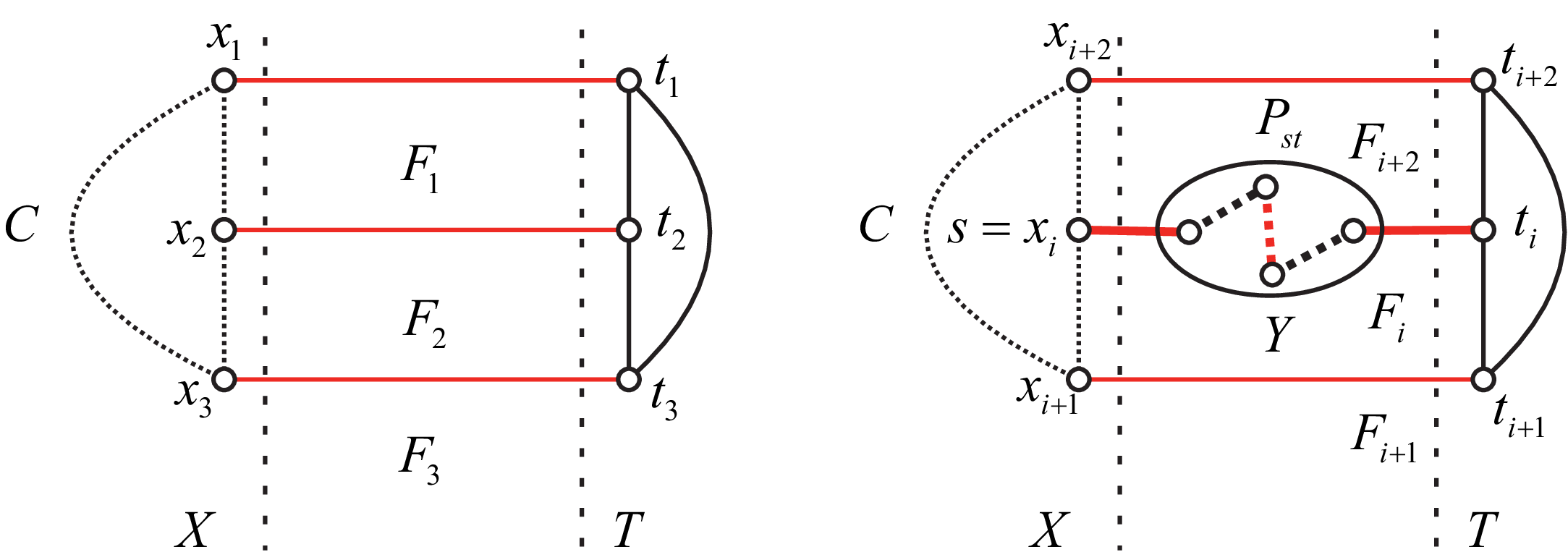}
         \caption{\label{fn9-1}Left: the graph $G_p^1$; Right: the graph $G_p^2$.}
      \end{figure}

      If $M_G^X\cap E_G[X,Y]=\emptyset$, then $e_i=x_it_i$ for $i=1,2,3$ and thus let $$G_p^1=G[E(C)\cup E(T)\cup \{e_1,e_2,e_3\}]$$ (with labels as shown in Fig. \ref{fn9-1}).
      Otherwise, $|M_G^X\cap E_G[X,Y]|=|M_G^X\cap E_G[\{s\},Y]|=1$ (by Eq. (\ref{eq3})). Let $M_G^X\cap E_G[\{s\},Y]=\{sy_s\}$ with $y_s\in Y$. Then $sy_s\in \{e_1,e_2,e_3\}$.
      Moreover, since $|M_G^X\cap E_G[V(T),Y]|=1$, $M_G^X\oplus M_Y$ contains an odd path $P_{st}$, whose end vertices are $s$ and some vertex $t_j$ ($j\in \{1,2,3\}$). In this case, let $$G_p^2=G[E(C)\cup E(T)\cup \{e_1,e_2,e_3\}\cup E(P_{st})]$$ (with labels as shown in Fig. \ref{fn9-1}). In both cases, $M_G^X\cap E(G_p^j)$ ($j=1,2$) is a perfect matching of $G_p^j$. So $G_p^j$ is an even subdivision of a prism.
      Fix a plane embedding of $G$. Since $G_p^j$ is a subgraph of $G$, it inherits this embedding, which is unique by Corollary \ref{Bi3}. Hence we can speak of the faces of $G_p^j$.
      For $i=1,2,3$, let $F_i$ be the facial cycle of $G_p^j$ containing all of $\{x_i,t_i,x_{i+1},t_{i+1}\}$, where the subscripts are taken modulo 3. Then each $F_i$ ($i=1,2,3$) is an $M_G^X$-alternating cycle.

      Since $G/X/T$ is 3-connected, $|N_{G/X/T}(x)\cap Y|\geq 2$ and $|N_{G/X/T}(t)\cap Y|\geq 2$.
      Since $N_G(s)\cap Y=N_{G/X/T}(x)\cap Y$ (deduced from Eq. (\ref{eq3})),
      \begin{equation}
      \label{neq5}
      |N_G(s)\cap Y|\geq 2.
      \end{equation}
            Since $|M_{G/X}^T\cap E_{G/X}[V(T),\{x\}\cup Y]|=3$ (by Eq. (\ref{eq1})),
            \begin{equation}
               \label{neq6}
            |M_{G/X}^T\cap E_G[V(T), Y]|\geq 2.
            \end{equation}
\begin{Subcla}
   \label{subcla-1}
  $M_Y\oplus M_{G/X}^T$ contains a unique path $P_t$ whose two end vertices belong to $V(T)$.
Moreover, $|N_G(V(T)\setminus V(P_t))\cap X|\leq 1$.
\end{Subcla}
\begin{figure}[h]
   \centering
   \includegraphics[scale=0.3]{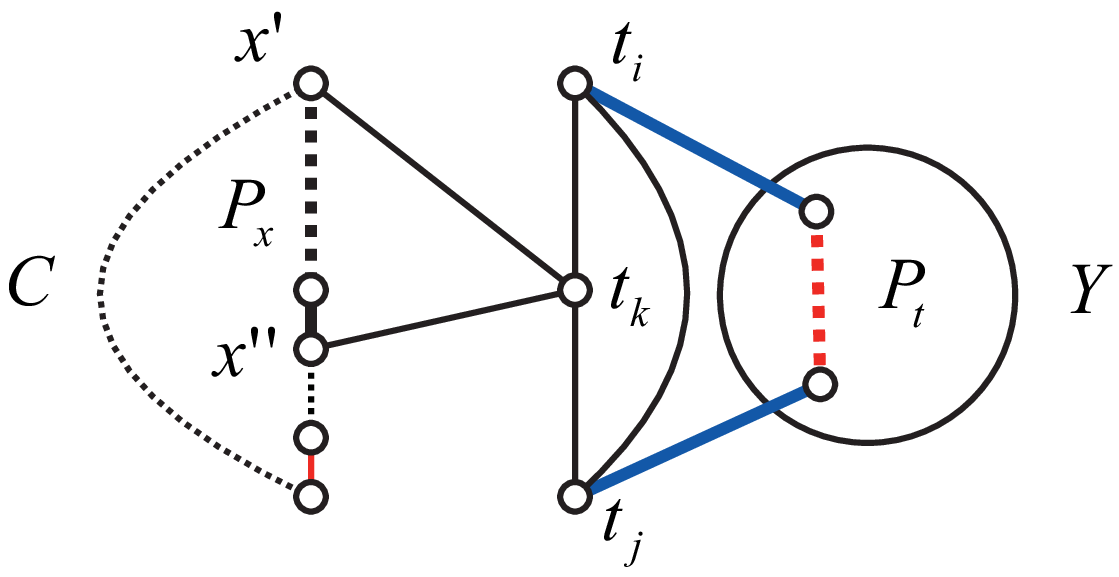}
   \caption{\label{fn9-0}Illustration for the proof of Claim \ref{subcla-1}.}
\end{figure}
\begin{proof}
 The symmetric difference $M_Y\oplus M_{G/X}^T$ has exactly four vertices of degree one, namely $x,t_1,t_2,t_3$, and thus has two path components (and possibly alternating cycles). Exactly one path has both ends in $V(T)$; call it $P_t$. Its first and last edges belong to $M_{G/X}^T$, so it has odd length (see Fig. \ref{fn9-0}).
Let $t_i$ and $t_j$ be two end vertices of $P_t$ and $\{t_k\}=V(T)\setminus V(P_t)$, where $\{i,j,k\}=\{1,2,3\}$. Suppose that $N_G(t_k)\cap X$ contains at least two distinct vertices, say $x'$ and $x''$.
Since $C$ is an odd cycle, it contains an even path $P_x$ with end vertices $x'$ and $x''$.
Thus $C-V(P_x)$ has a perfect matching, say $M_C$ (which may be empty).
Then $P_x+x't_kx''$ and $P_t+t_it_j$ are two vertex-disjoint even cycles.
Removing these two cycles yields a subgraph with a perfect matching $M_C\cup (M_Y\setminus E(P_t))$, contradicting the fact that $G$ is PM-compact.
\end{proof}

  \begin{Subcla}
   \label{NC-2}
   $\overline{x}$ is the hub of the underlying wheel of $G/\overline{X}$.
   \end{Subcla}
   \begin{proof}
          Suppose to the contrary that $\overline{x}$ is not the hub of the underlying wheel of $G/\overline{X}$.
         Then $\overline{x}$ is adjacent to exactly three vertices in $X$, say $x_1$, $x_2$ and $x_3$, and thus $E_G[X,Y]$ is covered by $\{x_1, x_2, x_3\}$.
         Since $E_G[X,Y]=E_G[\{s\},Y]$ (by Eq. (\ref{eq3})), $s\in \{x_1,x_2,x_3\}$.

         First we will show that $s$ is the hub of the underlying wheel of $G/\overline{X}$. If not, then relabel the vertices so that $x_2$ is the hub of the underlying wheel of $G/\overline{X}$ and $s=x_3$ (with labels as shown in Fig. \ref{f4-2} $(a)$).
         Note that both $G[X]-x_1$ and $G[X]-x_3$ contain Hamilton cycles, say $C_{x_1}$ and $C_{x_3}$, respectively.
         Since $G[\{x_3\}\cup Y]\cong G'-t$ and $G'$ is 3-connected, $G[\{x_3\}\cup Y]$ is 2-connected.
         Since $|M_{G/X}^T\cap E_G[V(T), Y]|\geq 2$ (by Ineq. (\ref{neq6})),
         $G[V(T)\cup \{x_3\}\cup Y]$ is 2-connected. Let $e\in E_G[\{x_3\},Y]$ (by Ineq. (\ref{neq5})).
         Since $e\in E(G/X)$ and $G/X$ is matching covered, $G/X$ has a perfect matching $M_e$ containing $e$.
         Since $G-V(C_{x_3})=G[V(T)\cup \{x_3\}\cup Y]$ and $M_e$ is a perfect matching of
         $G[V(T)\cup \{x_3\}\cup Y]$, $G-V(C_{x_3})$ has at least two different perfect matchings by Theorem \ref{BE}, a contradiction to the PM-compactness of $G$. So $s$ is the hub of the underlying wheel of $G/\overline{X}$ and thus $s=x_2$.

         \begin{figure}[h]
            \centering
         \includegraphics[scale=0.3]{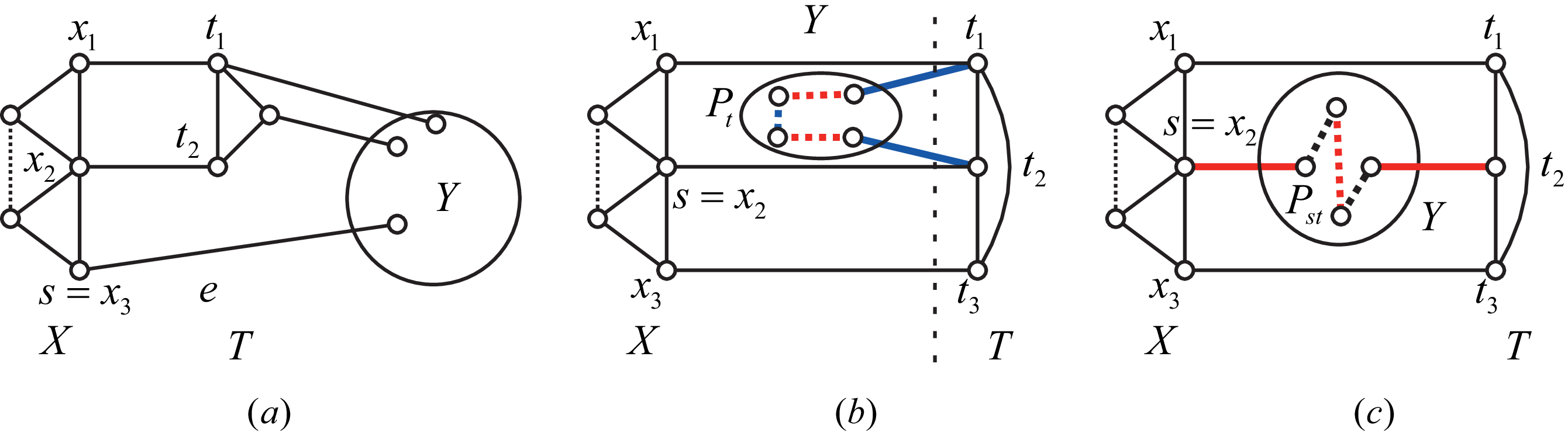}
            \caption{\label{f4-2}Illustration for Claim \ref{NC-2}.}
         \end{figure}

         Since $(N(x_1)\cup N(x_3))\cap \overline{X}\subseteq V(T)$ (by Eq. (\ref{eq3})), $e_i=x_it_i$ for $i=1,3$, and thus $V(e_2)\cap \overline{X}\subseteq \{t_2\}\cup Y$.
         Recall that $M_Y\oplus M_{G/X}^T$ contains an odd path $P_t$, and   two end vertices of $P_t$ belong to $V(T)$ (by Claim \ref{subcla-1}). Next we will show that $t_2\in V(P_t)$.

         If $x_2t_2\in M_G^X$, then we consider the plane embedding of the graph $G_p^1$.
          Since $s=x_2$, $N(x_2)\cap Y\neq \emptyset$ by Ineq. (\ref{neq5}). By planarity and the connectivity of $G[Y]$, it is embedded in the face bounded by $F_1$ or $F_2$ of $G_p^1$ (with labels as shown in Fig. \ref{f4-2} $(b)$). As $V(P_t)\cap Y\neq\emptyset$, the internal vertices of $P_t$ lie in this face and its ends lie on its boundary. Hence $t_2\in V(P_t)$.
         If $x_2y_2\in M_G^X$ for some $y_2\in Y$, then we consider the plane embedding of the graph $G_p^2$.
         In this case, $P_{st}$ is a common boundary of $F_1$ and $F_2$ in the graph $G_p^2$ (see Fig. \ref{f4-2} $(c)$).
         If $t_2\notin V(P_t)$, then $t_1$ and $t_3$ are the two end vertices of $P_t$. Since $G[Y]$ is connected and contains the internal vertices of $P_{st}$, it lies in the union of the closures of $F_1$ and $F_2$. The two boundary arcs of this region from $t_1$ to $t_3$ pass through $X$ and $t_2$, respectively. Thus the cycle $P_t+t_1t_3$ separates $t_2$ from $X$, and $N_G(t_2)\cap X=\emptyset$ by planarity.
         It implies that $xt_2\notin E(G/X)$ and thus $M_Y\oplus M_{G/X}^T$ contains an odd path $P_t'$, whose end vertices are $x$ and $t_2$ in the graph $G/X$.
          Since $P_t$ and $P_t'$ are distinct path components of $M_Y\oplus M_{G/X}^T$, they are vertex-disjoint. By Eq. (\ref{eq3}), this contradicts the separation of $t_2$ from $X$ by $P_t+t_1t_3$.
         So $t_2\in V(P_t)$.

          Relabel the vertices of $T$ so that $t_1$ and $t_2$ are the two end vertices of $P_t$.
          Then deleting the vertex sets of $P_t+t_1t_2$ and $C_{x_3}$ leaves a subgraph with the perfect matching
         $(M_Y\setminus E(P_t))\cup \{x_3t_3\}$, a contradiction to the PM-compactness of $G$. So $\overline{x}$ is the hub of the underlying wheel of $G/\overline{X}$.
      \end{proof}

   Recall that any multiple edges of $G/\overline{X}$ are parallel to spokes. By Claim \ref{NC-2}, $G[X]=C$ and each vertex in $X$ has degree two in $G[X]$.

         \begin{Subcla}
         \label{subcla-3}
         For $1\leq i\leq 3$, if $|E(P_i)|>1$ and $s$ is not an internal vertex of $P_i$, let $P_i=x_ia_1a_2\dots a_{2q}x_{i+1}$, where the subscripts are taken modulo 3, then we have $\{a_1t_i,a_{2q}t_{i+1}\}\subseteq E(G)$.
      Moreover, both $t_i$ and $t_{i+1}$ have neighbors in $Y$.
      \end{Subcla}
      \begin{proof}
      We first show that $a_1t_i\in E(G)$.
      Since $d_G(a_1)\geq 3$ and $d_C(a_1)=2$, $E_G[\{a_1\},V(T)\cup Y]\neq \emptyset$.
      Since $s$ is not an internal vertex of $P_i$, $s\neq a_1$.
      Since $E_G[X\setminus \{s\}, Y]=\emptyset$ (by Eq. (\ref{eq3})),
      $E_G[\{a_1\},V(T)\cup Y]=E_G[\{a_1\},V(T)]$ and thus
      $E_G[\{a_1\},V(T)]\neq \emptyset$. Since $P_i$ and $t_it_{i+1}$ lie on the facial cycle $F_i$ of $G_p^j$, $j\in \{1,2\}$, either $a_1t_i\in E(G)$ or $a_1t_{i+1}\in E(G)$ by the planarity of $G$. If $a_1t_{i+1}\in E(G)$, then
      all the vertices in $\{a_1,t_{i+1},x_{i+1}\}$ belong to a common facial cycle $F_0$ of $G_p^j+a_1t_{i+1}$ and $F_0$ is an $M_G^X$-alternating cycle (see Fig. \ref{fn9-2}). The facial cycle $F_{i+2}$, the subscript $i+2$ is taken modulo 3, is an $M_G^X$-alternating cycle vertex-disjoint from $F_0$, a contradiction to Prop. \ref{PM-P2}. So $a_1t_i\in E(G)$. Similarly, we can get that $a_{2q}t_{i+1}\in E(G)$.
\begin{figure}[h]
         \centering
         \includegraphics[scale=0.3]{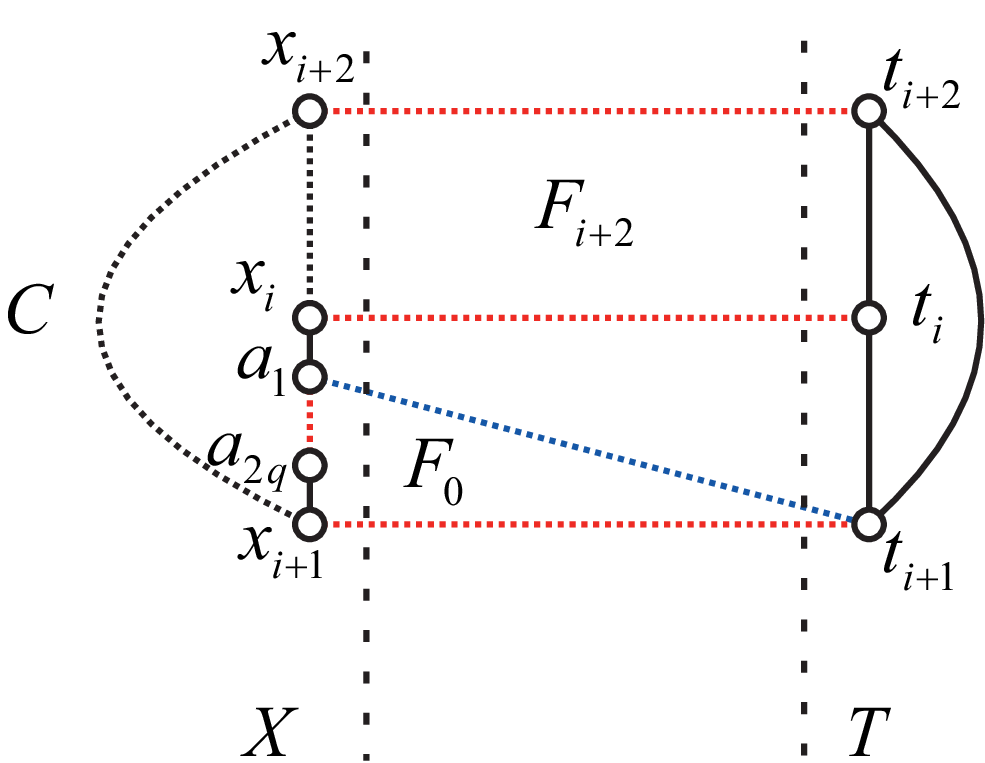}
         \caption{\label{fn9-2}Illustration for the case that $a_1t_{i+1}\in E(G)$ in Claim \ref{subcla-3}.}
      \end{figure}

      Suppose that at least one of $\{t_i,t_{i+1}\}$ has no neighbors in $Y$, say $t_i$.
      Then $e_i=x_it_i$.
      Recall that $M_Y\oplus M_{G/X}^T$ contains an odd path $P_t$, and the end vertices of $P_t$  belong to $V(T)$.
      Then $t_i\notin V(P_t)$.
      By Claim \ref{subcla-1}, $|N_G(t_i)\cap X|\leq 1$, contradicting $\{a_1,x_i\}\subseteq N_G(t_i)\cap X$.
      So both $t_i$ and $t_{i+1}$ have neighbors in $Y$.
      \end{proof}
      \begin{Subcla}
         \label{subcla-4}
         $|M_G^X\cap E_G[X,Y]|=1$.
      \end{Subcla}
      \begin{proof}
      Suppose $|M_G^X\cap E_G[X,Y]|\neq 1$. By Eq. (\ref{eq3}), $|M_G^X\cap E_G[X,Y]|\leq 1$, and thus $|M_G^X\cap E_G[X,Y]|=0$.
      Then $M_G^X\cap \partial_G(X)=M_G^X\cap E_G[X,V(T)]=\{e_1,e_2,e_3\}$.
      Recall that $G_p^1=G[E(C)\cup E(T)\cup \{e_1,e_2,e_3\}]$, and
      $F_i=P_i+x_it_it_{i+1}x_{i+1}$ ($i=1,2,3$) is a facial cycle of $G_p^1$, where the subscripts are taken modulo 3.
      Since $G'$ is 3-connected, $G'-t$ is 2-connected and $G'-x-t$ is connected. Note that $G[Y]=G'-x-t$. Since $\displaystyle\sum_{i=1}^3|E(P_i)|=|X|\geq 5$, we may cyclically relabel the indices so that $|E(P_3)|\geq 3$. Let $P_3=x_3a_1a_2\dots a_{2q}x_1$ with $a_0=x_3$, $a_{2q+1}=x_1$ and $q\geq 1$.
      \begin{figure}[h]
         \centering
         \includegraphics[scale=0.3]{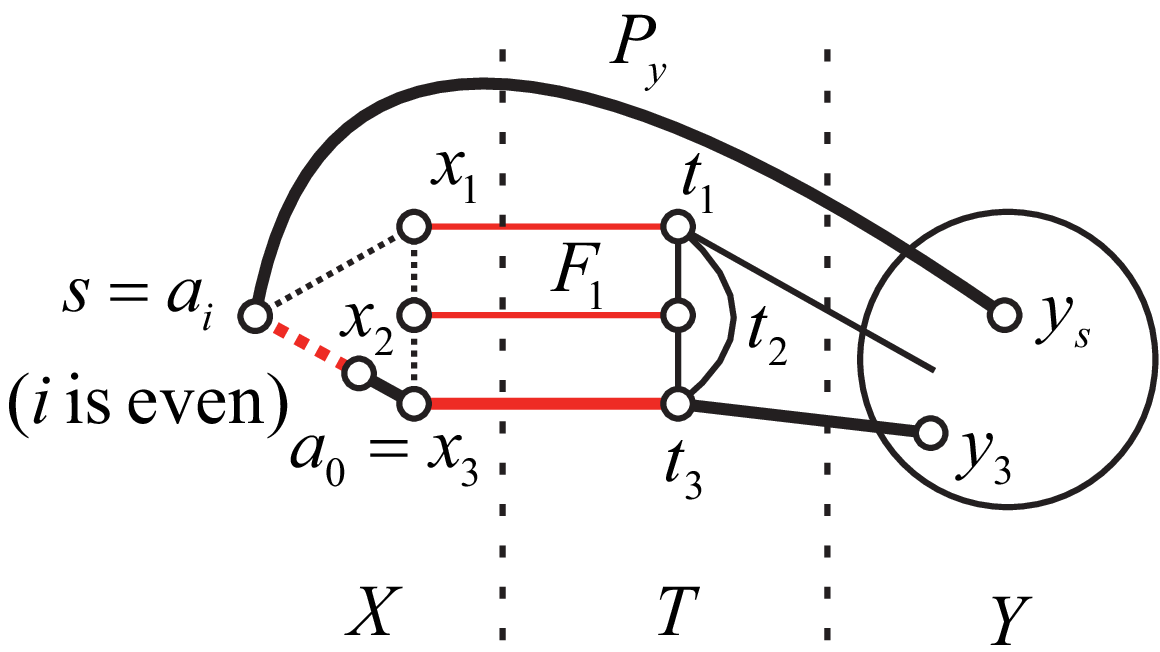}
         \caption{\label{f4-1}Illustration for the proof of Claim \ref{subcla-4}.}
      \end{figure}

      If both $t_1$ and $t_3$ have neighbors in $Y$,
      then by planarity, $G[Y]$ is embedded in the face bounded by $F_3$ of $G_p^1$ (see Fig. \ref{f4-1}), and thus
      $N_G(Y)\subseteq V(P_3)\cup \{t_1,t_3\}$.
      Since $|V(P_3)|$ is even, we can assume that
       $s=a_i$ for some even integer $0\leq i\leq 2q+1$ (by symmetry).
      Let $y_3\in N_G(t_3)\cap Y$.
      Since $|N_G(s)\cap Y|\geq 2$ (by Ineq. (\ref{neq5})), $s\in V(P_3)$ and $s$ has a neighbor $y_s$ in $Y$ distinct from $y_3$. Let $P_y=y_3t_3x_3a_1a_2\dots a_iy_s$. Then $P_y$ is an odd $M_G^X$-alternating path with end vertices $y_3$ and $y_s$.
      Note that all internal vertices of $P_y$ belong to $\overline{Y}$ and $P_y$ is vertex-disjoint from $F_1$.
      Since $E_G[X,Y]=E_G[\{s\},Y]$ (by Eq. (\ref{eq3})),
      $G[\{s\}\cup Y]=G'[\{x\}\cup Y]$.
      Since $G'-t$ is isomorphic to $G[\{s\}\cup Y]$ and it is 2-connected, $G[\{s\}\cup Y]$ is 2-connected.
      Since the end vertices of $P_y$ are $y_3$ and $y_s$ and $y_3\neq y_s$, $G[\{s\}\cup Y]+P_y$ is 2-connected. Since $s\in V(P_y)$, $G[\{s\}\cup Y]+P_y$ is a spanning subgraph of the graph $G[V(P_y)\cup Y]$.
      So $G[V(P_y)\cup Y]$ is 2-connected. Since $M_G^X\cap E[V(P_y)\cup Y]$ is a perfect matching of $G[V(P_y)\cup Y]$,
      the graph $G[V(P_y)\cup Y]$ contains at least two perfect matchings by Theorem \ref{BE}.
      Since $F_1$ is an $M_G^X$-alternating cycle vertex-disjoint from the graph $G[V(P_y)\cup Y]$,
      $G-V(F_1)$ contains at least two perfect matchings, a contradiction to the PM-compactness of $G$.
      So either $t_1$ or $t_3$ has no neighbors in $Y$.
      By Claim \ref{subcla-3}, $s$ is an internal vertex of $P_3$, i.e., $s\in V(P_3)\setminus \{x_1,x_3\}$.

       By Ineq. (\ref{neq6}), at least two vertices in $T$ have neighbors in $Y$. By the preceding discussion, $t_2$ must have a neighbor in $Y$. By planarity, $G[Y]$ is embedded in the face bounded by $F_1$ or $F_2$ of $G_p^1$, and thus each internal vertex of $P_3$ has no neighbors in $Y$. This contradicts $s\in V(P_3)\setminus \{x_1,x_3\}$ and $N_G(s)\cap Y\neq \emptyset$. So $|M_G^X\cap E_G[X,Y]|=1$.
      \end{proof}

      \begin{Subcla}
         \label{subcla-5}
         There is an integer $i\in \{1,2,3\}$ such that $s=x_i$, $|E(P_i)|=|E(P_{i-1})|=1$ and $t_i\in V(P_t)$,
         where the subscripts are taken modulo 3.
      \end{Subcla}
      \begin{proof}
      Since $|M_G^X\cap E_G[X,Y]|=1$ (by Claim \ref{subcla-4}) and $E_G[X,Y]=E_G[\{s\},Y]$ (by Eq. (\ref{eq3})), $|M_G^X\cap E_G[\{s\},Y]|=1$. Let $M_G^X\cap E_G[\{s\},Y]=\{sy_s\}$ with $y_s\in Y$. Recalling that $M_G^X\cap \partial_G(X)=\{e_1,e_2,e_3\}$, $sy_s\in \{e_1,e_2,e_3\}$.
      After relabelling the indices, assume that $sy_s=e_2$. Then $s=x_2$ and both $P_1$ and $P_2$ contain the vertex $s$. We next show that $|E(P_1)|=|E(P_2)|=1$.

      Suppose that $|E(P_2)|>1$.
      Let $P_2=x_2a_1a_2\dots a_{2q}x_3$ with $q\geq 1$.
       Since $s$ is not an internal vertex of $P_2$, $a_1t_2\in E(G)$ and both $t_2$ and $t_3$ have neighbors in $Y$ by Claim \ref{subcla-3}.
      Let $G_p=G[E(C)\cup E(T)\cup \{e_1,e_3,a_1t_2\}]$. Then $G_p$ is a subdivision of a prism. By Corollary \ref{Bi3}, $G_p$ has a unique plane embedding.
      Recall that $P_{st}$ is a path with end vertices $s$ and $t_2$.
       The internal vertices of $P_{st}$ lie in the face bounded by $P_1+x_1t_1t_2a_1x_2$ of $G_p$. Since $G[Y]$ is connected, it is embedded in that face (see Fig. \ref{fn11} ($a$)). By planarity, $N_G(Y)\subseteq V(P_1)\cup \{t_1,t_2,a_1\}$. So $N_G(t_3)\cap Y=\emptyset$, a contradiction to Claim \ref{subcla-3}. Thus $|E(P_2)|=1$. Similarly, $|E(P_1)|=1$.

      \begin{figure}[h]
         \centering
         \includegraphics[scale=0.3]{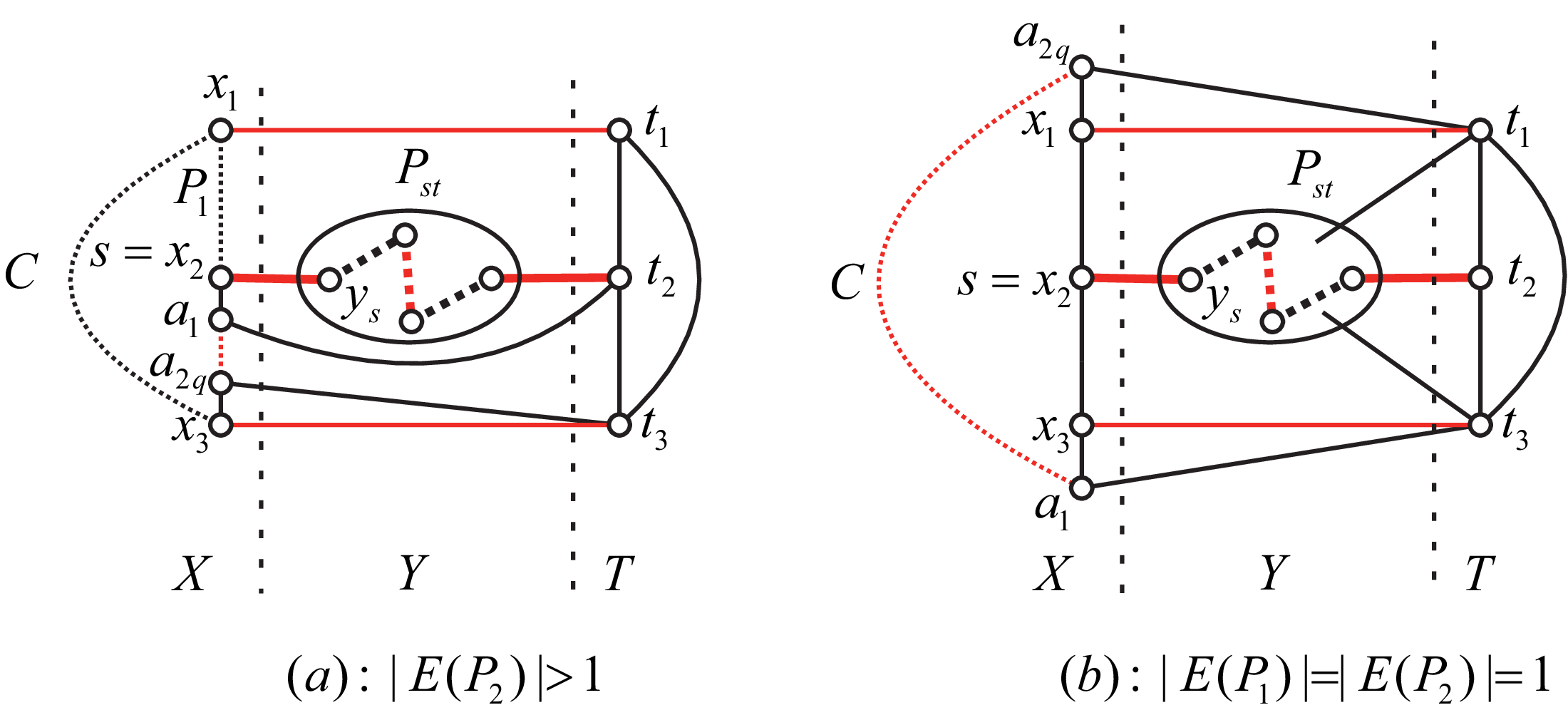}
         \caption{\label{fn11}Illustration for the proof of Claim \ref{subcla-5}.}
      \end{figure}
      Here $s=x_2$, $e_2=sy_s$, and $e_i=x_it_i$ for $i=1,3$. Thus the $G_p^2$ separation argument in the $x_2y_2\in M_G^X$ case of the proof of Claim \ref{NC-2} (Fig. \ref{f4-2}(c)) applies and gives $t_2\in V(P_t)$.
      \end{proof}

      By Claim \ref{subcla-5}, we may relabel the vertices and paths so that
      $s=x_2$, $|E(P_1)|=|E(P_2)|=1$ and the end vertices of $P_t$ are $t_2$ and $t_3$.
      Then $|E(P_3)|>1$, $t_1\notin V(P_t)$ and $e_1=x_1t_1$.
       By Claim \ref{subcla-3}, $P_3$ has an internal vertex adjacent to $t_1$.
      Then $|N_G(t_1)\cap X|\geq 2$, a contradiction to Claim \ref{subcla-1}.
      From the preceding discussion, $G/T$ is a brick.

   \subsection{Proof of Claim \ref{cla7}}
Recall that $Y=V(G)\setminus (X\cup V(T))$ and $T=t_1t_2t_3t_1$ is a triangle of $G/X\rightarrow x$. Since $x\in V(T)$, relabel its vertices so that $t_3=x$ in $G/X$.
Then $N_G(t_i)\cap X\neq \emptyset$ for $i=1,2$.

\begin{Subcla}
   \label{subcas2-1}
   If $|N_G(\{t_1,t_2\})\cap X|=1$, then $G/T'\rightarrow t'$ is a brick,
   where $T'=x't_1t_2x'$ and $\{x'\}=N_G(\{t_1,t_2\})\cap X$.
\end{Subcla}
   \begin{proof}
      Since $V(T')\subseteq \overline{Y}$, $G/T'$ is a splicing of $(G/T')/Y$ and $G/\overline{Y}$.
      Since $x\in V(T)$, $G/\overline{Y}=(G/X)/T=G'$. Since $G'$ is a brick, so is $G/\overline{Y}$.
      By Theorem \ref{2.1}, it suffices to show  $(G/T')/Y$ is also a brick and $\partial_{G/T'}(Y)$ cannot be covered by a vertex in $Y$ and a vertex in $(X\setminus \{x'\})\cup \{t'\}$.

      To show that $(G/T')/Y$ is a brick, we will show that $(G/T')/Y$ and $G/\overline{X}$ have the same underlying graph, i.e., an odd wheel. Since $G/X$ is a brick, $N_{G}(\{t_1,t_2\})\cap Y\neq \emptyset$.
      It follows that $(G/T')/Y$ can be obtained from $G[X]$ by adding the vertex $y$, and edge set $E_G[X\setminus\{x'\},Y]\cup \{ x'y\}$ (multiple edges allowed).
      Since $E_G[X, \{t_1,t_2\}]=E_G[\{x'\}, \{t_1,t_2\}]$, $E_G[X\setminus \{x'\}, \{t_1,t_2\}]=\emptyset$. So
      $E_G[X\setminus \{x'\}, \overline{X}]=E_G[X\setminus \{x'\}, \{t_1,t_2\}\cup Y]=E_G[X\setminus \{x'\},Y].$
      Therefore, the underlying graphs of $(G/T')/Y$ and $G/\overline{X}$ are the same.

      \begin{figure}[h]
      \centering
      \includegraphics[scale=0.35]{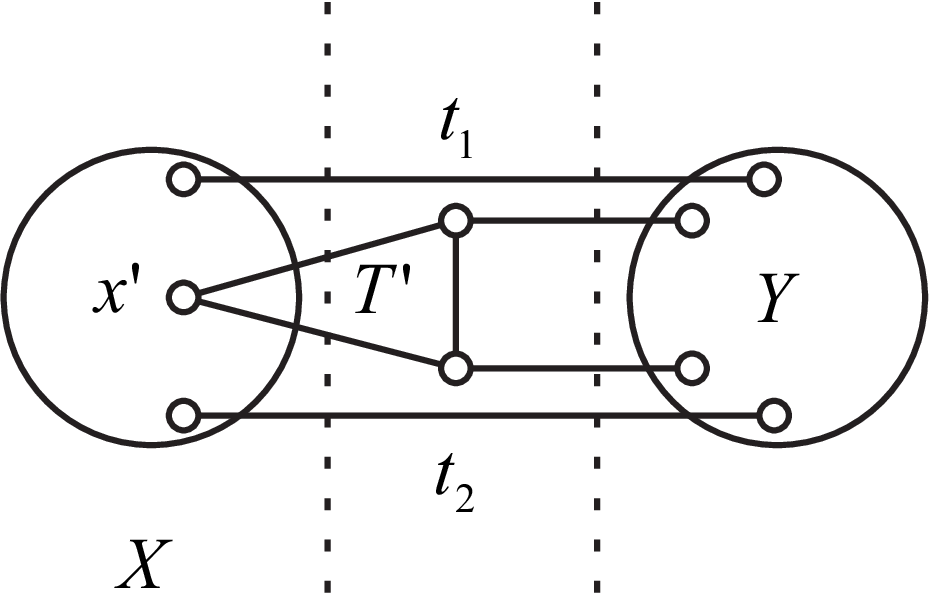}
      \caption{\label{f3-4}Illustration for the proof of Claim \ref{subcas2-1}, where $x=t_3$.}
      \end{figure}
       The three edges of $M_G^X\cap\partial_G(X)$ give a matching of size three across $\partial_{G/T'}(Y)$: if one is $x't_i$ for $i\in\{1,2\}$, replace it by the $M_G^X$-edge from the other vertex of $\{t_1,t_2\}$ to $Y$; otherwise retain all three. Thus no two vertices cover the cut.
      By Theorem \ref{2.1}, $G/T'$ is a brick.
   \end{proof}
   Next we only consider the case that
$|N_G(\{t_1,t_2\})\cap X|\geq 2$.
      \begin{Subcla}
         \label{subcas2-2}
         If $|N_G(\{t_1,t_2\})\cap X|\geq 2$, then $G/Y\rightarrow y$ is a PM-compact brick.
      \end{Subcla}
      \begin{proof}
      Since $X\cap Y=\emptyset$ and $Y\subseteq \overline{X}$, $G/Y$ is a splicing of $G/\overline{X}$ and $(G/Y)/X=(G/X)/Y$, whose underlying graphs are an odd wheel and $K_4$, respectively.
      If $\partial_{G/Y}(X)$ is covered by a vertex in $X$ and a vertex in $\{y,t_1,t_2\}$,
      then such two vertices form a 2-vertex cut of $G/Y$.
      Since $G$ is 3-connected, each 2-vertex cut of $G/Y$ contains the vertex $y$.
      It implies that $G[X\cup \{t_1,t_2\}]$ is not 2-connected.
      Since $N_G(t_i)\cap X\neq \emptyset$ for $i=1,2$ and $|N_G(\{t_1,t_2\})\cap X|\geq 2$, $E_G[\{t_1,t_2\},X]$ contains a matching of size two. Since $G[X]$ is 2-connected, so is $G[X\cup \{t_1,t_2\}]$, a contradiction.
      By Theorem \ref{2.1}, $G/Y$ is a brick.
      Since $G/\overline{Y}$ is also a brick, $\partial_G(Y)$ is a separating cut of $G$.
      By Lemma \ref{SC}, $G/Y$ is PM-compact.
      \end{proof}
We will consider two cases depending on whether $\overline{x}$ is the hub of the underlying wheel of $G/\overline{X}$.
\begin{Cas}
   \label{ca4-1}

   $\overline{x}$ is the hub of the underlying wheel of $G/\overline{X}$.
\end{Cas}

By Lemma \ref{WkPath}, $G/Y$ contains a spanning subgraph $H$ such that $H\cong G^{\triangle}$ (see Fig. \ref{f2-0}). It implies that $G[X]$ is a simple odd cycle with length five, say $x_1x_2x_3x_4x_5x_1$, and $yt_1t_2y$ is a triangle of $H$, which contains exactly two vertices with degree four in $H$.
By symmetry, interchange $t_1$ and $t_2$ if necessary so that $d_H(t_1)=4$.
According to the structure of $G^{\triangle}$, $t_1$, together with two vertices from $X$, forms a triangle $T_1$ of $H$.
Since $X\cup \{t_1\}\subseteq V(G)$, $T_1$ is also a triangle of $G$.
Relabel the vertices of $X$ so that $T_1=t_1x_1x_2t_1$.
      \begin{Subcla}
         \label{cla7.2}
         If $d_H(t_2)=4$, then $G[Y]$ is a simple triangle such that $G/Y$ is a brick. If $d_H(y)=4$, then $G/T_1$ is a brick.
      \end{Subcla}
      \begin{proof}
      If $d_H(t_2)=4$, then $d_H(y)=3$. Relabel the remaining vertices of $X$ so that $x_3y\in E(G/Y)$ and $t_2x_4x_5t_2$ is a triangle of $G/Y$ (see the left graph in Fig. \ref{f3-5}).
      Next we consider those edges incident with $y$ in $G/Y$.
      Since none of $x_3y$, $t_1y$ and $t_2y$ is a forcing edge of $G/Y$ and $G/Y$ is PM-compact (by Claim \ref{subcas2-2}), none of them is a multiple edge of $G/Y$ by Lemma \ref{MF}, and such three edges must be forcing edges of $G/\overline{Y}$ by Lemma \ref{GF}.
      It implies that each of $x_3$, $t_1$ and $t_2$ has only one neighbor in $Y$ in the graph $G$.
      By the planarity of $G$, neither $x_1$ nor $x_5$ has neighbors in $Y$.
      If $x_2y\in E(G/Y)$ (resp. $x_4y\in E(G/Y)$),
      then it is not a forcing edge of $G/Y$ and thus a forcing edge of $G/\overline{Y}$ by Lemma \ref{GF}.
      By Lemma \ref{MF}, $x_2y$ (resp. $x_4y$) is not a multiple edge of $G/Y$ and thus $x_2$ (resp. $x_4$) has exactly one neighbor in $Y$ for the graph $G$.
       Note that $\{x_2y,x_4y\}$ is not an edge subset of $G/Y$; otherwise, $yx_2x_3x_4y$ and $t_1x_1x_5t_2t_1$ are two vertex-disjoint even cycles containing all the vertices of $G/Y$, contradicting the PM-compactness of $G/Y$.
      By the preceding discussion, $d_{G'}(\overline{y})\leq 4$. Every edge at $\overline{y}$ is forcing, so $|\mathcal{M}(G')|=d_{G'}(\overline{y})$. Since $G'$ is a brick, Lemmas \ref{MG3} and \ref{MG4} imply that its underlying graph is $K_4$ (the triangular prism is cubic). The forcing property also excludes parallel edges in $G[Y]$. Hence $G[Y]$ is a simple triangle.
      Since $|N_G(\{t_1,t_2\})\cap X|\geq 2$, $G/Y$ is a brick by Claim \ref{subcas2-2}.

      \begin{figure}[h]
         \centering
         \includegraphics[scale=0.3]{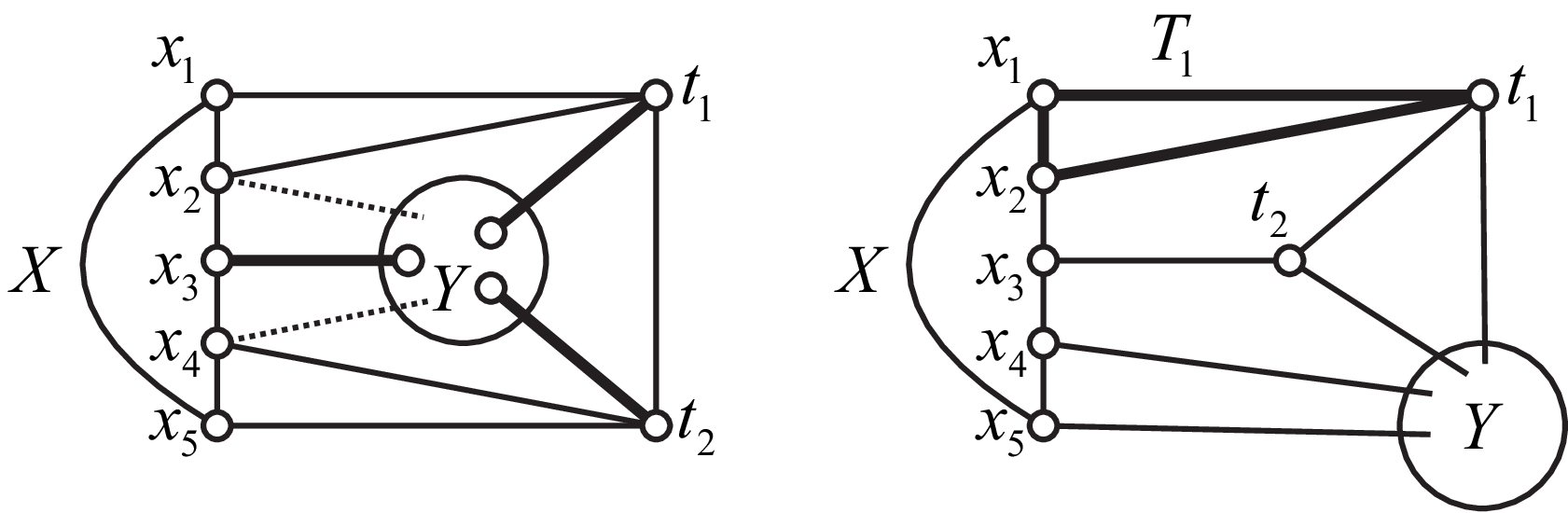}
         \caption{\label{f3-5}Illustration for the proof of Claim \ref{cla7.2}.}
      \end{figure}

      If $d_H(y)=4$, then $d_H(t_2)=3$. Relabel the remaining vertices of $X$ so that $x_3t_2\in E(G/Y)$ and $yx_4x_5y$ is a triangle of $G/Y$ (see the right graph in Fig. \ref{f3-5}).
      We can contract triangles three times from $G$ to obtain $G/\overline{Y}=G'$: let $G_1^*=G/T_1\rightarrow t_1^*$ and $G_2^*=G_1^*/\{t_1^*,t_2,x_3\}\rightarrow t_2^*$. Then $G'=G_2^*/\{t_2^*,x_4,x_5\}\rightarrow x$.
      Since $G'$ is a brick, it can be checked that both $G_2^*$ and $G/T_1$ are also bricks by Theorem \ref{2.1}.
   \end{proof}
      \begin{Cas}
         \label{ca4-2}

         $\overline{x}$ is not the hub of the underlying wheel of $G/\overline{X}$.
      \end{Cas}
      In this case, $\partial_G(X)$ is covered by exactly three vertices $x_1$, $x_2$ and $x_3$ in $X$, and
      we may relabel the vertices so that $x_2$ is the hub of the underlying wheel of $G/\overline{X}$.
      Since $G/Y$ is a brick (by Claim \ref{subcas2-2}) and $\partial_{G/Y}(X)$ is nontrivial, $\partial_{G/Y}(X)$ is not a tight cut of $G/Y$.
      Since $\partial_{G/Y}(X)=\partial_G(X)$, $G/Y$ contains a perfect matching $M_{G/Y}^X$ such that $|M_{G/Y}^X\cap \partial_{G/Y}(X)|=3.$
      Recall that $M_{G/X}^T$ is a perfect matching of $G/X$ such that $|M_{G/X}^T\cap \partial_{G/X}(V(T))|=3$ and $x=t_3$.
      Let $M_{G/X}^T\cap \partial_{G/X}(V(T))=\{xy_0, t_1y_1, t_2y_2\}$, where $y_i\in Y$ for $i=0,1,2$.
      Since $xy_0\in E(G/X)$ and $\partial_G(X)$ is covered by $\{x_1,x_2,x_3\}$,
       the edge $xy_0$ corresponds to $x_1y_0$, $x_2y_0$, or $x_3y_0$ in $G$.

      \begin{Subcla}
         \label{Subcla2.4}
         For the hub $x_2$ of the underlying wheel of $G/\overline{X}$, we have $x_2y\in M_{G/Y}^X$.
      \end{Subcla}
      \begin{proof}
         Suppose that $x_2y\notin M_{G/Y}^X$. By symmetry, after relabelling, $\{x_1y,x_2t_1,x_3t_2\}\subseteq M_{G/Y}^X$.
         First we will show that $x_2y_0$ (in $G$) corresponds to the edge $xy_0$ (in $G/X$).
         Since $x_3\overline{x}$ is not a spoke of the underlying wheel of $G/\overline{X}$,
         it is not a multiple edge of $G/\overline{X}$.
         Since $x_3t_2\in E(G)$, $N_G(x_3)\cap \overline{X}=\{t_2\}$ and thus $x_3y_0\notin E(G)$.
         Since $G/\overline{Y}=(G/X)/T$, $G/\overline{Y}$ is bicritical and thus $G[Y]$ is factor-critical (the deletion of each vertex results in a matchable graph).
         Let $M_0$ be the perfect matching of $G[Y]-y_0$.
         Then $M_{G/X}^T\oplus M_0$ contains an alternating path $P_t$ with end vertices $t_1$ and $t_2$.
         Since the edges of $P_t$ incident with its end vertices belong to $M_{G/X}^T$, $|E(P_t)|$ is odd and thus $P_t+t_1t_2$ is an even cycle.
         Since $\overline{x}$ and $x_1$ are two adjacent vertices on the rim of $G/\overline{X}$, $G[X]-x_1$ has a Hamilton cycle, say $C$.
         If $x_1y_0\in E(G)$ (see the left graph in Fig. \ref{f3-6}), then $C$ and $P_t+t_1t_2$ are two vertex-disjoint even cycles whose removal leaves a subgraph that has a perfect matching $\{x_1y_0\}\cup (M_0\setminus E(P_t))$, a contradiction to the PM-compactness of $G$.
         So $x_1y_0\notin E(G)$ and thus $x_2y_0$ corresponds to the edge $xy_0$.
         \begin{figure}[h]
            \centering
            \includegraphics[scale=0.3]{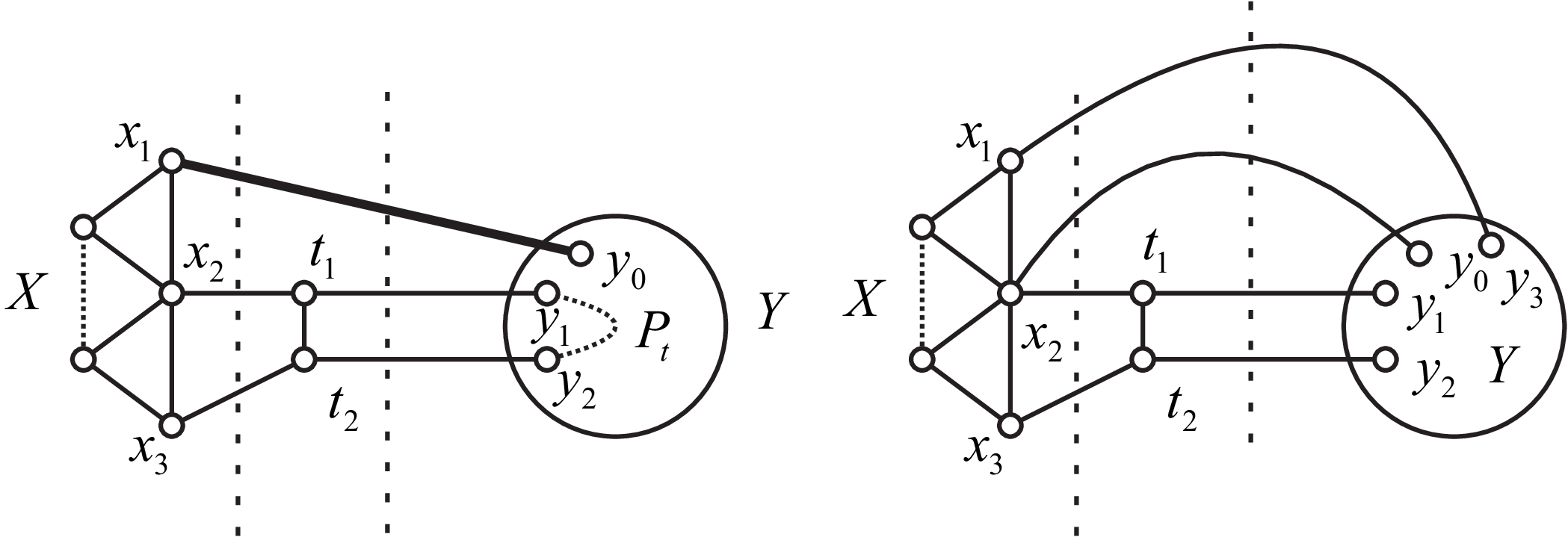}
            \caption{\label{f3-6}Illustration for the proof of Claim \ref{Subcla2.4}.}
         \end{figure}

         As $x_1\overline{x}$ is not a spoke of the underlying wheel of $G/\overline{X}$, it is not a multiple edge of $G/\overline{X}$.
         Since $x_1y\in E(G/Y)$, the unique neighbor of $x_1$ belongs to $Y$, say $y_3$ (see the right graph in Fig. \ref{f3-6}).
       Since $G/\overline{Y}$ is a brick, it is 3-connected and thus $G[Y]$ is 2-connected. Since $G$ is planar, the boundary of every face of $G[Y]$ is a cycle. The subgraph $G[\overline{Y}]$ is embedded in a face of $G[Y]$; let $F$ be the boundary cycle of this face.
       We may assume that $y_0$, $y_1$, $y_2$ and $y_3$ appear on $F$ in anticlockwise order (allowing $y_3=y_2$).
      Since $x_1y_0\notin E(G)$ and $x_1y_3\in E(G)$, $y_3\neq y_0$.
      Let $M_3$ be a perfect matching of $G[Y]-y_3$.
       Then $M_{G/X}^T\oplus M_3$ contains two vertex-disjoint alternating paths $P_1$ and $P_2$ in $G[Y]$,
      whose end vertices belong to $\{y_0,y_1,y_2,y_3\}$.
      Since the alternating path containing $y_3$ has the edges incident with its end vertices belonging to $M_{G/X}^T$ and $M_3$ respectively, say $P_1$, the number of its edges is even.
      Since $P_1$ and $P_2$ are disjoint and $G[Y]$ is planar, another end vertex of $P_1$ is either $y_0$ or $y_2$ (if $y_3=y_2$, then $P_1$ is a single vertex $y_2$).
      Since $y_3\notin V(P_2)$, both end vertices of $P_2$ belong to $M_3$ and thus $|E(P_2)|$ is odd.
      Let $P_x'$ be the rim path of the underlying wheel of $G/\overline{X}$ from $x_1$ to $x_3$ avoiding $\overline{x}$, and set $P_x=P_x'+x_3x_2$. Then $P_x'$ spans $G[X]-x_2$ and has an odd number of edges, whereas $P_x$ spans $G[X]$, runs from $x_1$ to $x_2$, and has an even number of edges.
      If the end vertices of $P_1$ are $y_0$ and $y_3$, then those of $P_2$ are $y_1$ and $y_2$. The cycles $C_1=P_x+x_1y_3+P_1+y_0x_2$ and $C_2=P_2+y_1t_1t_2y_2$ are vertex-disjoint and even. Deleting them leaves a subgraph with the perfect matching $M_3\setminus(E(P_1)\cup E(P_2))$.
      If the end vertices of $P_1$ are $y_2$ and $y_3$, then the end vertices of $P_2$ are $y_0$ and $y_1$.
       The cycles $P_x'+x_1y_3+P_1+y_2t_2x_3$ and $P_2+y_1t_1x_2y_0$ are vertex-disjoint and even. Deleting them leaves a subgraph with the perfect matching $M_3\setminus(E(P_1)\cup E(P_2))$.
      This contradicts the PM-compactness of $G$.
      \end{proof}

      \begin{Subcla}
         \label{cla7.3}
         $G[Y]$ is a simple triangle such that $G/Y$ is a brick.
      \end{Subcla}
      \begin{proof}
         By Claim \ref{Subcla2.4}, $x_2y\in M_{G/Y}^X$.
         By symmetry, interchange $t_1$ and $t_2$ if necessary so that $M_{G/Y}^X\cap \partial_{G/Y}(X)=\{x_1t_1,x_3t_2,x_2y\}$.
         For the graph $G/Y$, we consider all edges incident with $y$. Note that $N_{G/Y}(y)\subseteq \{x_1,x_2,x_3,t_1,t_2\}$.
         By the discussion in Claim \ref{Subcla2.4}, each of $x_1$ and $x_3$ has a unique neighbor in $\overline{X}$, respectively, i.e., $t_1$ and $t_2$.
         Thus $x_1y\notin E(G/Y)$ and $x_3y\notin E(G/Y)$.
         Since $\{t_1y_1, t_2y_2\}\subseteq M_{G/X}^T$ and $x_2y\in M_{G/Y}^X$,
         $\{t_1y,t_2y,x_2y\}\subseteq E(G/Y)$.
         Since none of the edges in $\{t_1y,t_2y,x_2y\}$ is a forcing edge of $G/Y$ and $G/Y$ is PM-compact (by Claim \ref{subcas2-2}),
         none of them is a multiple edge of $G/Y$ by Lemma \ref{MF}, i.e., each vertex in $\{t_1,t_2,x_2\}$ has a unique neighbor in $Y$ for the graph $G$.
          By Lemma \ref{GF}, the three corresponding edges incident with the contraction vertex $\overline{y}$ in $G/\overline{Y}$ are forcing. They are all its incident edges, and every perfect matching contains exactly one of them. Hence $G/\overline{Y}$ has exactly three perfect matchings.
         Since $G/\overline{Y}$ is a brick, $G/\overline{Y}\cong K_4$ by Lemma \ref{MG3}, and thus $G[Y]$ is a simple triangle.
         Since $|N_G(\{t_1,t_2\})\cap X|\geq 2$, $G/Y$ is a brick by Claim \ref{subcas2-2}.
      \end{proof}

       If $|N_G(\{t_1,t_2\})\cap X|=1$, then $G$ contains a triangle $T'$ such that $G/T'$ is a brick by Claim \ref{subcas2-1}.
     When $|N_G(\{t_1,t_2\})\cap X|\geq 2$, if the contraction vertex $\overline{x}$ is the hub of the underlying wheel of $G/\overline{X}$,
      then either $G[Y]$ is a simple triangle such that $G/Y$ is a brick, or $G$ contains a triangle $T_1$ such that $G/T_1$ is a brick by Claim \ref{cla7.2};
      otherwise, $G[Y]$ is a simple triangle such that $G/Y$ is a brick by Claim \ref{cla7.3}. Therefore, the result follows.
\qed

\section*{Acknowledgements}
This work was supported by
the National Natural Science Foundation of China (Nos. 12271235, 12671409 and 12601685),
the Startup Research Fund for Introduced Talents of Sichuan Normal University (No. kyqd20260308),
the Natural Science Foundation of Fujian Province (No. 2026J002034),
the Institute of Meteorological Big Data-Digital Fujian,
and Sichuan Science and Technology Program (No. 2026NSFSC2032).

\par\medskip
\noindent\begin{minipage}{\linewidth}
\textbf{Conflict of interest}\par\smallskip
The authors declare that they have no conflict of interest.
\end{minipage}

\par\medskip
\noindent\begin{minipage}{\linewidth}
\textbf{Data availability}\par\smallskip
No datasets were generated or analysed during the current study.
\end{minipage}

\par\medskip
\noindent\begin{minipage}{\linewidth}
\textbf{Declaration on the use of AI tools}\par\smallskip
No AI tools have been used in the generation of the mathematical ideas of this paper. AI tools assisted with manuscript preparation and review. The authors take full responsibility for the mathematical content.
\end{minipage}

\end{document}